\documentclass[12pt]{article}
\usepackage{amsmath,amsfonts,amsthm}
\usepackage{color}
\usepackage{hyperref}
\hypersetup{colorlinks=true,
            linkcolor={black},
            citecolor={falured}}
\theoremstyle{definition}
\newtheorem{assumption}{Assumption}
\usepackage{ulem}
\usepackage{amsmath,amssymb,fullpage}
\usepackage{enumerate}
\usepackage{graphicx}
\usepackage{harvard}
\usepackage{makeidx}
\usepackage[utf8]{inputenc}
\usepackage{wrapfig}
\usepackage{mathrsfs}
\usepackage{amssymb}
\usepackage{latexsym}
\usepackage{amsbsy}
\usepackage{color}
\usepackage{bbm}
\usepackage{mathrsfs}
\usepackage{url}

\usepackage{wasysym}
\def\email#1{\it #1\par}

\providecommand{\otherindexspace}[1]{}

\usepackage{amsthm}
\newtheorem{theorem}{Theorem}[section]
\newtheorem{lemma}[theorem]{Lemma}
\newtheorem{proposition}[theorem]{Proposition}
\newtheorem{remark}[theorem]{Remark}
\newtheorem{definition}[theorem]{Definition}

\newtheorem{corollary}[theorem]{Corollary}

\DeclareMathOperator{\supp}{supp}

\DeclareMathOperator{\esssup}{esssup}
\DeclareMathOperator*{\argmax}{arg\,max}

\definecolor{falured}{rgb}{0.5, 0.09, 0.09}

\numberwithin{equation}{section}

\title{Mean-field games of optimal stopping: \\a relaxed solution approach}
\author{G\'eraldine Bouveret \thanks{Smith School, University of Oxford, South Parks Road, Oxford, OX1 3QY, United Kingdom, Email: \email  geraldine.bouveret@smithschool.ox.ac.uk}\and Roxana Dumitrescu \thanks{Department of Mathematics, King's College London, Strand, London, WC2R 2LS, United Kingdom, Email: \email roxana.dumitrescu@kcl.ac.uk} \and Peter Tankov \thanks{ENSAE Paris, 5 avenue Henry Le Chatelier, 91120 Palaiseau, France, Email: \email peter.tankov@polytechnique.org}}

\date{}

\begin{document}
\maketitle
\begin{abstract}
We consider the mean-field game where each agent determines the
optimal time to exit the game by solving an optimal stopping problem
with reward function depending on the density of {the state processes of} agents still
present in the game. We place ourselves in the framework of relaxed
optimal stopping, which amounts to look{ing} for the optimal occupation
measure of the stopper rather than the optimal stopping time. This
framework allows us to prove the existence of a relaxed
Nash equilibrium and the uniqueness of the associated value of the
representative agent under mild assumptions. Further, we prove a rigorous relation between
relaxed Nash equilibria and the notion of mixed solutions introduced in earlier
works on the subject, and provide a criterion, under which the
optimal strategies are pure strategies, that is, behave in a similar
way to stopping times. Finally, we present a numerical method
for computing the equilibrium in the case of potential games and show its convergence.

\end{abstract}


 \textbf{AMS:} 91A55, 91A13, 60G40

\section{Introduction}

\quad The purpose of this paper is to study a large-population
stochastic differential game of optimal stopping, where each agent
finds the optimal time to exit the game by solving an optimal stopping
problem with instantaneous reward function depending on the density of
{the state processes of} agents still present in the game. To motivate the mean-field game
(MFG) framework, we first provide a formulation with
a finite number of agents. Assume that {each agent} $i=1,2,...,{N}$ {has a private state process} $X^{i}$, {whose dynamics} {is} given by the stochastic differential equation (SDE),
\begin{align*}
dX^i_t = \mu(t,X^i_t) dt + \sigma(t,X^i_t) dW^i_t,
\end{align*}
where the Brownian motions $W^i$, $i=1,\dots,N$ are independent.

The objective of each agent $i$ is to maximize over all possible
stopping times $\tau$ the reward functional
$$
\mathbb E\left[\int_0^\tau e^{-\rho t} \tilde f(t,X^i_t, m^n_t)
dt + e^{-\rho (\tau \wedge T)} g(\tau\wedge T, X^i_{\tau \wedge T})\right ],
$$
with
\textcolor{black}{
$$
m^N_t(dx) = \frac{1}{N} \sum_{i=1}^N \delta_{X^i_t}(dx) \mathbf 1_{t
\leq {\tau^{i}}},
$$
where $\tau^{i}$ represents the optimal stopping time of the agent $i$.}
{A}gents have the same state process coefficients and objective functions, and the optimal stopping problems are coupled only {through} the empirical measure $m^N$. Since the objective functions are coupled, it is natural to look for Nash equilibria.

Stochastic differential games with a large number
{$n$} of players are rarely tractable. The MFG approach amounts to
looking for a Nash equilibrium in the limiting regime, when the
number of players {$n$} goes to infinity. Following this approach, we
study the \textit{{MFG} of optimal stopping}, which can be
seen as an {infinite}-agent version of the above game. {In this
  approach, we first solve for a fixed {\color{black}flow of sub-probability
  measures $(m_t)_{0\leq t\leq T}$}} the optimal stopping problem
$$
\max_\tau \mathbb E\left [\int_0^\tau e^{-\rho t} \tilde f(t,X_t, m_t)
dt + e^{-\rho (\tau \wedge T)} g(\tau\wedge T, X_{\tau \wedge T})\right],
$$
with
$$dX_t = \mu(t,X_t) dt + \sigma(t,X_t) dW_t.$$
Then, {given $\tau^{m,x}$ the optimal stopping time for the agent with initial condition $x$, and the initial measure $m_0$, we look for} the {\color{black}flow
of
measures}  $(m_t)_{0\leq t \leq T}$ such that
\begin{align}
m_t(A) = \int m_0(dx)\mathbb P[X^{x}_t\in A; \tau^{m,x} >t],
\quad A \in \mathcal B(\mathbb R^d), \ t\in [0,T].\label{part2}
\end{align}

\textcolor{black}{Note that in $\eqref{part2}$, the probability is
  \textit{not} a conditional but joint probability.}
A solution (Nash equilibrium) of the MFG problem
is the {\color{black}flow of measures} $(m_t)_{0\leq t\leq T}$, which is the fixed point of the mapping defined
by the right-hand side of \eqref{part2}.

In this paper, we  prove {the} existence of the Nash
equilibrium {for} the MFG problem and the uniqueness of the associated value of the
representative agent. {To this aim, we use} the \textit{relaxed solution}
approach, which converts the stochastic optimal stopping problem into
a linear programming problem over a space of
measures. The decision variable is no longer the \textit{optimal
  stopping time}, but rather the distribution of the killed state process.

Introducing relaxed solutions facilitates existence proofs: the
existence is proven by using Fan-Glicksberg's fixed\textcolor{black}{-}point theorem. The
relaxed solutions are related to {the} \textit{mixed strategies} introduced
in {\cite{bertucci2017optimal}},  and we establish a rigorous relation
between the two. Finally, we propose an implementable numerical scheme
for computing a Nash equilibrium in the case of potential games, and
show its convergence. An application of these results to a resource{-}sharing problem will be developed in a companion paper.

{MFG} theory has been  introduced by P.-L. Lions and
J.-M. Lasry in a series of papers
\cite{lasry2006jeux,lasry2006jeux1,lasry2007mean} using an analytic
approach  and studied independently at about the same time by \cite{huang2006large}. Later on, a
probabilistic approach has been  developed in a series of papers by
Carmona, Delarue, and their co-authors \cite{carmona2013probabilistic,carmona2013mean,carmona2018probabilistic,carmona2016mean,lacker2015mean} and so on.

The analytic method consists in finding the Nash equilibria through a
coupled  system of nonlinear partial differential equations: a
Hamilton-Jacobi-Bellman equation (backward in time), which describes
the optimal control problem of the \textit{representative} agent when
the distribution $\mu$ is given, and a Kolmogorov-type equation
(forward in time) which describes the evolution of the density under
the optimal control. In the probabilistic approach, the system of PDEs
is replaced by a coupled system of forward-backward stochastic
differential equations of McKean-Vlasov type.

{MFGs} of optimal stopping have been considered in the
literature only very recently, and our understanding of this type of
games remains limited. \cite{nutz2018mean}  considers a MFG problem
where the agents interact through the proportion of players that have
already stopped and each agent solves a specific optimal stopping problem of the form
$$\sup_\tau \mathbb{E}\left[\exp\left(\int_0^\tau r_sds\right)\textbf{1}_{\{\theta>\tau\}\cup{\{}\theta=\infty{\}}}\right].$$
{There}, the process $r$ creates an incentive for the agent to stay in
the game, while the possibility of default at a random time $\theta$
creates an incentive to leave. The distribution of $\theta$ depends on
the proportion $\rho_t$ of players who have already stopped in such {a}
way that the departure of other agents creates an incentive for the
agent under consideration to leave as well {\color{black}(this type of game is known
as preemption game).}
In a similar spirit but with greater generality,
\cite{carmona2017mean} consider {MFGs} of timing, whose
formulation is motivated by a dynamic model of bank run in {a continuous}
time setting. As in \cite{nutz2018mean}, the payoff of each agent
depends on the proportion of players who have already stopped, and the
departure of players creates an additional incentive for the players
still in the game to leave as well. Both papers (\cite{nutz2018mean}
and \cite{carmona2017mean}) adopt a purely probabilistic
approach.

In contrast to these two references,  \cite{bertucci2017optimal}
studies a {MFG} of optimal stopping, which is similar to the
one considered in this paper, {i.e.} where the interaction takes place through the
density of states of agents remaining in the game, rather than the
proportion of players that have already stopped. In this reference and
in our
paper, the departure of players creates an incentive for the players still in the game to
stay, {\color{black}a type of behavior known as 'war of attrition',} which is characteristic of resource-sharing
problems. In \cite{bertucci2017optimal} the state process has constant
coefficients and evolves in a bounded domain, and the {MFG}
of optimal stopping is solved through a coupled system of a
Hamilton-Jacobi-Bellman variational inequality and a Fokker-Planck
equation.

\cite{bertucci2017optimal} makes a number of significant contributions to the
literature. {In particular, he provides} an example of non-existence of
\textit{Nash equilibrium with pure strategies} in optimal stopping
MFG, and {introduces} the notion of mixed strategies in this context, for which
existence may be recovered. However, the existence proofs in this
paper are not fully clear to us.\footnote{To be precise, the weak convergence of the flow
  $m^\varepsilon$ established in the proof of existence of a mixed solution in both stationary and
  parabolic cases (Theorems 1.6 for the stationary case and Theorem
  2.1 for the parabolic case) is not sufficient to conclude that
  $\int f(m^\varepsilon) dm^\varepsilon$ converges.
}
To clarify the existence question and solve the MFG {of} optimal stopping
problem in greater generality (with variable coefficients and in
unbounded domains), we adopt, {in this paper,} a completely different
approach, based on the relaxed solution technique.

The approach of \textit{relaxed solutions/controls} is a relatively popular method of
compactification of stochastic control problems to establish existence of
solutions, which comes in several different flavors. In, e.g.,
\cite{el1987compactification} and a number of other papers, the
authors reformulate the control problem as a relaxed controlled
martingale problem. A similar approach is used by
\cite{lacker2015mean} in the context of (standard) MFG.  In the
second approach, especially popular for infinite-horizon and ergodic
control problems, the control problem is reformulated as a linear
programming problem on the space of measures, and one looks for the
joint occupation density of the state process and the control. {We refer the reader to, e.g.,
\cite{buckdahn2011stochastic} and \cite{stockbridge1990time}, for a
link between these two formulations}. \textcolor{black}{The literature on \textit{relaxed
  solutions} for individual optimal stopping problems is quite
limited. \cite{SC2002} propose a linear programming formulation for
the infinite-horizon optimal stopping of a Markov diffusion process,
using two measures: the occupation measure of the process and the
joint distribution of the stopping time and the stopped
process. \cite{HS} extend this result to processes with singular
components such as reflected diffusions. In contrast to these two
references, in our paper we propose a different formulation based only
on the occupation measure of the process killed at the stopping time. 
To the best of our
knowledge, ours is the first paper which uses \textit{relaxed
  solutions} in order to solve optimal stopping problems of mean-field type.}

\textcolor{black}{The literature on numerical schemes for MFG
  is well developed in the case of MFG
  with regular controls (see e.g. \cite{BC2015}), but very little is
  known in the case of MFG with optimal stopping. In the latter case
  \cite{B2018} proposes an algorithm, which works only under the
  assumption that the instantaneous reward function is
  \textit{strictly monotonic} with respect to the measure, which is quite restrictive for applications. We propose instead a different algorithm, which allows to consider the case of a \textit{non-strictly} monotonic reward function.}

The structure of the paper is the following. {I}n Section 2, we present
the model and give the mean-field formulation of the problem. In
Section 3, we introduce the relaxed formulation of the single-agent
optimal stopping problem and establish the existence of a relaxed
solution. In Section 4, we study the relaxed optimal stopping problem
in the MFG context and give conditions for the existence of a Nash
equilibrium and uniqueness of {the} Nash equilibrium {value}. In Section 5,
we establish the relation between the relaxed and strong formulation of both single-agent and MFG optimal stopping problems. {Finally, in} Section 6, we present the numerical algorithm and provide convergence results.
%

\section{The model}\label{se model}

We fix a terminal time horizon $T<\infty$, and introduce a possibly
unbounded open domain $\textcolor{black}{{\mathcal O}\subseteq \mathbb R^d}$ on which the state
processes of the agents will evolve. The space of bounded positive
measures on ${\mathcal O}$ will be denoted by $\mathcal{M}({\mathcal O})$, and
the space of probability measures on $\mathcal O$ will be denoted by
$\mathcal P({\mathcal O})$. 
\textcolor{black}{In the sequel, any element $x \in \mathbb{R}^d$ will be identified to a column vector with $i$-th component $x^{i}$ and Euclidian norm $|\!|x|\!|.$ Similarly, for any matrix $A \in \mathbb{R}^{d\times K}$ we denote by  $|\!|A|\!|$ its Euclidian norm.}
\textcolor{black}{\paragraph{N-players game formulation} Consider $N$ agents whose states $X^i$, $i=1,\dots,N$} follow the
diffusion-type dynamics
\begin{align}
dX^i_t = \mu(t,X^i_t) dt + \sigma(t,X^i_t) dW^i_t,\quad X^i_0 = x^i\in
  \mathcal O,\label{sde}
\end{align}
where the \textcolor{black}{$K$-dimensional Brownian motions $W^i$, $i=1,\dots,N$ are independent} and
the coefficients $\mu$ and $\sigma$ satisfy the following assumption.
\begin{assumption}[\textbf{X-SDE}]
\textcolor{black}{The coefficients $\mu:[0,T] \times \mathcal{O} \mapsto
  \mathbb{R}^d$ and $\sigma:[0,T] \times \mathcal O \mapsto \mathbb{R}^{d \times K}$ are assumed to be Lipschitz
continuous in the second variable, uniformly in $t\in [0,T]$ and bounded.}
\end{assumption}
By classical results on SDEs, this assumption guarantees the
existence of a strong solution to \eqref{sde} satisfying

\textcolor{black}{$$\sup_{0\leq t\leq T} \mathbb E[|\!|X^i_t|\!|^p] <\infty, \textcolor{black}{\text{ for all } p \geq 1.}$$}


We
denote by $\mathcal L$ the infinitesimal generator of this process
$$
\textcolor{black}{\mathcal L f (t,x)= \nabla_X f(t,x)^\top\mu(t,x) + \frac{1}{2} Tr[\sigma^\top(H_Xf)\sigma],}
$$
\textcolor{black}{with $ \nabla_X f:=(\partial_{x_1}f,...,\partial_{x_d}f)^\top$, $H_Xf$ the Hessian matrix of $f$ with respect to $x$ and $Tr$ the trace operator.}

Each agent aims to determine the optimal stopping time $\tau_i$ valued in $[0,T]$ by
solving the optimal stopping problem
$$
\max_{\tau }\mathbb E\left[\int_0^{\tau\wedge \tau_{\mathcal O}^i} e^{-\rho t} \tilde f(t,X^i_t, m^N_t)
dt + e^{-\rho (\tau\wedge \tau_{\mathcal O}^i \wedge T)} g(\tau\wedge \tau_{\mathcal O}^i\wedge T, X^i_{\tau \wedge \tau_{\mathcal O}^i\wedge T})\right],
$$
where $\rho>0$ is a discount factor, $\tilde f: [0,T]\times
{\mathcal O}\times \mathcal M({\mathcal O})\to \mathbb R$ is the running reward function,
$g:[0,T]\times{\mathcal O}\to\mathbb R $ is the terminal reward, \textcolor{black}{$m^N_t$} is defined
by
\begin{align}\label{empiricalmeas}
m^N_t(dx) = \frac{1}{\textcolor{black}{N}} \sum_{i=1}^{\textcolor{black}{N}} \delta_{X^i_t}(dx) \mathbf 1_{t
\leq \textcolor{black}{\tau^{i}} \wedge \tau_{\mathcal O}^i},
\end{align}
\textcolor{black}{with $\tau^{i}$ a stopping time with respect to the filtration
generated by the Brownian motions of all agents, corresponding to agent $i$} and $\tau_{\mathcal O}^i$ the exit time from the domain $\mathcal O$ of agent $i$.
The assumptions on $\tilde f$ will be specified later, and $g$ is
assumed to belong to $C^{1,2}([0,T]\times {\mathcal O})$ and has
derivatives of order $1$ in $t$ and of orders $1$ and $2$ in $x$ of
polynomial growth in $x$ uniformly in $t$.
Letting $f(t,x,\mu) = e^{-\rho t} (\tilde f(t,x,\mu) - \rho g(t,x) +
\frac{\partial g}{\partial t} + \mathcal L g),$
the optimal stopping problem becomes (up to a constant),
\begin{align}\label{optimization}
\max_{\tau }\mathbb E\left[\int_0^{\tau\wedge \tau_{\mathcal O}^i} f(t,X^i_t, m^n_t)dt\right].
\end{align}

\textcolor{black}{We now formulate the notion of Nash equilibrium for the optimal stopping game with $N$ players. To this purpose,  \textcolor{black}{let $\mathcal{T}$ be the set of stopping times with respect to the filtration
generated by the Brownian motions of all agents, taking values between $0$ and $T$. Given} a strategy vector  ${\tau}:=(\tau^1, \tau^2,...,\tau^N) \textcolor{black}{\in \mathcal{T}^N}$ and an individual strategy $\sigma \textcolor{black}{\in \mathcal{T}}$, let  $[\tau^{-i}, \sigma]$ indicate the strategy vector
that is obtained from $\tau$ by replacing $\tau^{i}$, the strategy of player $i$, with $\sigma$.
\begin{definition}[Nash Equilibrium $N$-players game]
A strategy vector  ${\tau}:=(\tau^1, \tau^2,...,\tau^N) \textcolor{black}{\in \mathcal{T}^N}$ is called a \textit{Nash equilibrium for the $N$ players game}, if for every $i \in \{1,2,..,N\}$ and every $\sigma \textcolor{black}{\in \mathcal{T}}$, we have
$$J_N^{i}(\tau) \geq J_N^{i}([\tau^{-i}, \sigma]),$$
\end{definition}
where, for each $\theta \textcolor{black}{\in \mathcal{T}^N}$,
$$J_N^{i}(\theta):=\mathbb E\left[\int_0^{\theta^{i}\wedge \tau_{\mathcal O}^i} f(t,X^i_t, m^N_t)dt\right],$$
where $m_t^N$ is given by $\eqref{empiricalmeas}$ with $\tau^{i}$ replaced by $\theta^{i}$, for each $i$.}
\paragraph{MFG formulation}
In the limit of a large number
of agents, we expect, from the law of large numbers, that the
empirical measure \textcolor{black}{$m^N_t$} converges to a \textit{deterministic} limiting
distribution $m_t$ for each $t\in [0,T]$. The problem
of each agent therefore consists in finding the optimal stopping time
in the filtration generated by the individual noise of this agent
only, and it is sufficient to work on a probability space supporting a
single Brownian motion.

Let $(\Omega, \mathcal F, \mathbb P)$ be a probability space
supporting a \textcolor{black}{standard $K$-dimensional Brownian motion $W$}. We denote by $\mathbb F^W$
the natural filtration of $W$ completed with the sets of measure zero.
In the MFG formulation, the state of the representative
agent with initial value $x$ follows the dynamics
$$
dX^x_t = \mu(t,X^x_t) dt + \sigma(t,X^x_t) dW_t,
$$
where we write $X^x_\cdot$ as a shorthand for $X^{(0,x)}_\cdot$. As intimated in the introduction, the first step of the MFG approach consists in solving the following optimal stopping problem for the agent
\begin{align}
\max_{\tau \in \mathcal T^W([0,T])}\mathbb E\left[\int_0^{\tau\wedge
  \tau^{x}_{\mathcal O}} f(t,X^x_t, m_t)\label{optimaltime}
dt \right],
\end{align}
where $ \mathcal T^W([0,T])$ is the set of $\mathbb F^W$-stopping times
with values in $[0,T]$ and $\tau_{\mathcal O}^x\equiv \tau_{\mathcal O}^{(0,x)}$ is the exit time from the domain $\mathcal O$ of this agent with initial value $x$.
Then, given the optimal stopping time (solution of the
problem \eqref{optimaltime}) for the  agent with initial condition $x$, $\tau^{m,x}$, and the initial measure $m_0 \in \mathcal P({\mathcal O})$, the second step consists in finding the {\color{black}flow of measures}  $(m_t)_{0\leq t \leq T}$ such that
\begin{align}
m_t(A) = \int m_0(dx)\mathbb P[X^{x}_t\in A; \tau^{m,x}\wedge
  \tau^{x}_{\mathcal O} >t],
\quad A \in \mathcal B({\mathcal O}), \ t\in [0,T].\label{2ndstep}
\end{align}
In other words, the solution of the optimal stopping MFG problem is
the {\color{black}flow of measures} $(m_t)_{0\leq t\leq T}$,  which is the fixed point of the mapping defined
by the right-hand side of \eqref{2ndstep}. In the sequel, such
solution will be called a \emph{pure solution}. As shown in
\cite{bertucci2017optimal}, pure solutions for optimal stopping MFG problems do
not always exist, and for this reason in the sequel we shall consider
relaxed solutions. A relaxed solution is close in spirit to the mixed
solution introduced in \cite{bertucci2017optimal}, precise relationship between
the two notions will be established later in the paper.


\section{Relaxed formulation of the single-agent optimal stopping
  problem}
The relaxed formulation of the optimal stopping problem consists in finding the occupation measure of the
representative agent rather than the stopping time. We first provide a
relaxed formulation of the standard optimal stopping problem in this section and then
move to the relaxed formulation of the MFG problem in the following one. First, we
introduce the necessary notations.

\textcolor{black}{Let $V$ be the space of flows of (signed) bounded measures} on ${\mathcal O}$: $(m_t(\cdot))_{0\leq t \leq
  T} \in V$ is such that: for every $t\in [0,T]$, $m_t$ is a  \textcolor{black}{(signed) bounded
measure} on ${\mathcal O}$, for every $A\in \mathcal B({\mathcal O})$, the
mapping $t\mapsto m_t(A)$
is measurable, and $\int_0^T \int_{{\mathcal O}} m_t(dx)\, dt <\infty$. To each
flow $m\in V$, we associate a \textcolor{black}{signed measure} on $[0,T]\times {\mathcal O}$ defined by $\mu(dt,dx):= m_t(dx)\,
dt$. \textcolor{black}{The space $V$, endowed with the topology of weak
convergence (that is, $\int fd\mu^n \mapsto \int fd\mu$ for every function $f$ continuous and bounded) is a locally convex Hausdorff topological space (see e.g. \cite{V61}).}\\

Consider the optimal stopping problem
\begin{align}
\max_{\tau\in \mathcal T^W([0,T])} \mathbb E\left[\int_0^{\tau \wedge
  \tau^{x}_{\mathcal O}} f(t,X^x_t) dt\right],\label{optistop}
\end{align}
\textcolor{black}
{
In this section we study a relaxed version of this optimal stopping
problem, where the process $X$ starts with an initial distribution
$m^*_0 \in \mathcal P(\mathcal O)$ instead of a fixed value, and which is formulated in terms of
flows of measures rather than stopping times. We let $\bar m_t$ denote
the distribution of the process $X$, started with the initial distribution
$m_0^*$ and killed at the first exit time from $\mathcal O$. In other
words,
$$
\bar m_t(A) = \int_{\mathcal O} m_0^*(dx) \mathbb P[X^x_t\in A;
\tau^x_{\mathcal O}>t]. 
$$
We impose the following minimal assumption on the reward function
$f$. We shall see below in Corollary \ref{bounded.cor} that this
assumption is sufficient for the problem to be well defined, but
stronger assumptions will be imposed for existence of solution. 
{\color{black}\begin{assumption}[\textbf{$f$-min}]
The map $f:[0,T] \times \mathcal{O} \mapsto \mathbb{R}$ is measurable
and satisfies
\begin{align*}
\int_0^T\int_{\mathcal O}  (f(t,x))_- {\bar m}_t (dx)\, dt<\infty,
\end{align*}
where $()_-$ denotes the negative part. 
\end{assumption}}
}

{\color{black} The previous assumption was not
sufficient to guarantee that the integral in (3.2) is well
defined. 
}

\begin{definition}[Relaxed optimal stopping problem]\label{DEF}
For a given initial distribution $m^*_0\in \mathcal P({\mathcal O})$, the
relaxed formulation of the optimal stopping problem
\eqref{optistop} consists in finding the {\color{black}flow of measures}
$(m^*_t)_{0\leq t\leq T}$, which maximizes
the cost functional
\begin{align}
\int_0^T\int_{\mathcal O}  f(t,x) {m}_t (dx)\, dt,\label{relcost}
\end{align}
over $\hat m \in \mathcal A(m^*_0)$, where the set
$\mathcal A(m^*_0)\subseteq V$ contains all flows of positive bounded measures $(\hat
m_t)_{0\leq t\leq T}\in V$  satisfying
\begin{align}\label{constraint}
\int_{\mathcal O} u(0,x)m^*_0(dx)  + \int_0^T \int_{{\mathcal O}} \left\{\frac{\partial u}{\partial t} +
  \mathcal L u\right\} \hat m_t(dx) \, dt \geq 0,
\end{align}
for all $u \in C^{1,2}([0,T]\times {\mathcal O})$ such that $u\geq 0$ and $\frac{\partial
  u}{\partial t} + \mathcal{L}u$ is bounded on $[0,T]\times \mathcal O$.
\end{definition}

The rest of this section is devoted to the solution of the relaxed
optimal stopping problem. A precise connection with the strong (classical)
formulation of the optimal stopping problem will be established in
Section \ref{relation.sec}. To gain some intuition about this
definition right away, remark that for a stopping time $\tau\in
\mathcal T^W([0,T])$, we can introduce the occupation measure $m^x_t(A):= \mathbb E[\mathbf 1_{A}(X^x_t) \mathbf
1_{t\leq \tau\wedge \tau^{x}_{\mathcal O}}]$. Then the objective function of the optimal stopping problem writes
$$
\mathbb E\left[\int_0^{\tau\wedge \tau^{x}_{\mathcal O}} f(t,X^x_t) dt\right] = \int_{[0,T]\times {\mathcal O}}  f(t,y) m^x_t(dy)\, dt.
$$
On the other hand, by It\^o's formula, for a positive and regular test function $u$, one has
\textcolor{black}{
$$
u(0,x) + \int_{[0,T]\times {\mathcal O}} \left(\frac{\partial u}{\partial t} +\mathcal{L}u\right)(t,y)m_t(dy) \, dt = \mathbb E[u(\tau\wedge \tau^{x}_{\mathcal O}\wedge T, X_{\tau\wedge \tau^{x}_{\mathcal O}\wedge
T})] \geq 0.
$$}

In Lemmas \ref{apriori.lm} and \ref{compact.lm}, we study the
properties of the set $\mathcal
A(m^*_0)$. First note that this set is clearly nonempty since it contains the flow
$m_t(dx)\equiv 0$. To proceed, we need a regularity assumption on the
coefficients $\mu$ and $\sigma$. We distinguish two cases depending on
the type of boundary of $\mathcal O$.

\begin{assumption}[\textbf{X-PDE}]
The coefficients $\mu$ and $\sigma$ are such that for every $C^\infty$
bounded function $g:[0,T]\times {\mathcal O}\to \mathbb R$ {\color{black}with
bounded derivatives of all orders,} the equation
\begin{align}\label{PDE}
\left(\frac{\partial u}{\partial t} + \mathcal Lu\right)(t,x) = g(t,x)\quad (t,x)\in
[0,T]\times {\mathcal O},\quad u(T,x) = 0,\quad \forall x\in {\mathcal O},
\end{align}
has a $C^{1,2}$ solution $u$ on $[0,T]\times {\mathcal O}$ such that
$\frac{\partial u}{\partial x}$ has a polynomial growth in $x$,
uniformly in $t$, and such that one of the following two conditions
holds:
\begin{itemize}
\item[i.] The boundary of $\mathcal O$ is unattainable: for all $x\in
  \mathcal O$, $\tau^{x}_{\mathcal O} > T$ a.s.
\item[ii.] The solution $u$ belongs to $C([0,T]\times
  \overline{\mathcal O})$ and satisfies $u(t,x) = 0$ for $(t,x)\in
  [0,T]\times \partial \mathcal O$.
\end{itemize}
\end{assumption}

\begin{remark}\label{pde.rem}
Assumption (X-PDE) holds in a variety of different
settings. Below, while not aiming to give the sharpest possible
conditions, we present some examples of such settings. 
\begin{itemize}
\item 
  Let $\mathcal{O}=\mathbb{R}^d$ and assume that the operator $\mathcal L$ is uniformly
  parabolic: there exists $\gamma>0$ such that for all
  $(t,x) \in [0,T] \times \mathbb{R}^d$ and $\xi \in \mathbb{R}^d$,
  the $d\times d$ matrix $a=\sigma ^\top \sigma$ satisfies
\begin{align}\sum_{i,j=1}^d a_{i,j}(t,x)\xi^{i} \xi^{j} \geq \gamma
  |\xi|^2.\label{parabol}
  \end{align}
Furthermore, suppose that the coefficients $a_{ij}$ are bounded, uniformly H\"older continuous
in $x$ and uniformly continuous in $t$, and the coefficients $\mu_i$
are H\"older continuous in $x$ uniformly on compacts and continuous in
$t$. Then, by Theorem 4.4.6 in \cite{F75}, equation
\eqref{PDE} admits a $C^{1,2}$ solution, and  the polynomial growth of $\frac{\partial u}{\partial x}$ follows
from the estimate (4.4.12) in the above reference. 
\item  Let $\mathcal O$ be a bounded domain
  with $C^1$ boundary and assume that \eqref{parabol} is satisfied and the
  coefficients $a_{ij}$ and $\mu_i$ are uniformly H\"older continuous
  in $(t,x)$ on $[0,T]\times\mathcal O$. Then, by Theorem 4.3.6 in
  \cite{F75}  equation
\eqref{PDE} admits a $C^{1,2}$ solution. 
\item As our last example we consider a situation where the condition
  \eqref{parabol} need not be satisfied. For simplicity, we restrict
  ourselves to the setting of homogeneous equations, that is, the
  coefficients $\mu$ and $\sigma$ do not depend on $t$, but the
  argument may be extended to the general case.  Suppose that the
  boundary of $\mathcal O$ is unattainable and that 
  $\partial_{x_i} \mu$, $\partial^2_{x_i x_j}\mu$, $\partial_{x_i} \sigma$ and
  $\partial^2_{x_i x_j}\sigma$ are bounded and locally Lipschitz. 
This ensures that equation \eqref{sde}
  admits a unique strong solution,
  $$
  X^{x}_t  = \int_0^t \mu(X^x_s) ds + \int_0^t \sigma(X^x_s) dW_s,
  $$
  and, applying Theorem V.39 in \cite{protter} twice (first to the
  process $X$ and then to its first order tangent flow), we conclude
  that the mapping $x\mapsto X^{x}_t$ is twice continuously
  differentiable, and the derivatives $D^i_{kt}:=
  \partial_{x_i} X^{x}_{kt}$ and  $D^{ij}_{kt}:=
  \partial^2_{x_i x_j} X^{x}_{kt}$ are given by
  the solutions of the following system of equations (where we use the
  Einstein convention of summing over repeated indices and
  $\delta^i_k$ denotes the Kroneker symbol).
  \begin{align*}
    D^i_{kt} &= \delta^i_k  + \int_0^t \partial_{x_l}\mu^k(X^x_s) D^i_{ls}
               ds + \int_0^t \partial_{x_l} \sigma^{km}(X^x_s) D^i_{ls} dW^m_s,\\
     D^{ij}_{kt} &= \int_0^t \partial^2_{x_l x_n} \mu^k(X^x_s) D^i_{ls}D^j_{ls}
    ds+ \int_0^t \partial_{x_l} \mu^k(X^x_s) D^{ij}_{ls}
    ds \\ &+ \int_0^t \partial^2_{x_l x_n} \sigma^{km}(X^x_s) D^i_{ls}D^j_{ls} dW^m_s+ \int_0^t \partial_{x_l} \sigma^{km}(X^x_s) D^{ij}_{ls} dW^m_s.
  \end{align*}
  Moreover, by standard arguments (e.g., Theorem V.66 in
  \cite{protter} and Gronwall's lemma), from boundedness of
  derivatives of $\mu$ and $\sigma$ it follows that for some constant
  $K$,
  $$
  \max_{i,k}\mathbb E\left[\sup_{0\leq t\leq T} |D^i_{kt}|^p\right]+\max_{i,j,k}\mathbb E\left[\sup_{0\leq t\leq T} |D^{ij}_{kt}|^p\right]\leq K. 
  $$
Let us define
  $$
  u(t,x) = -\mathbb E\left[\int_t^T g(s,X^{(t,x)}_s) dt\right] =
  -\mathbb E\left[\int_0^{T-t} g(t+s,X^{(0,x)}_s) ds\right].  
  $$
  Then, by dominated convergence, the derivatives $\partial_t u$,
  $\partial x_i u$ and
  $\partial^2_{x_i x_j} u$ exist, are bounded, continuous,
  and given by the following expressions.  
  \begin{align*}
  \partial_t u(t,x) &= -\mathbb E\left[
    g(T,X^{(0,x)}_{T-t}) +\int_0^{T-t} 
                                       \partial_t g(t+s,X^{(0,x)}_s)
                                       ds\right] \\
   \partial_{x_i} u(t,x) &= -\mathbb E\left[
    \int_0^{T-t} \partial_{x_k} g(t+s,X^{(0,x)}_s)
                                         D^i_{ks}
                                         ds\right] \\
\partial^2_{x_i x_j} u(t,x) &= -\mathbb E\left[
    \int_0^{T-t} \left\{\partial^2_{x_k x_l}
                                       g(t+s,X^{(0,x)}_s)
                                         D^i_{ks} D^j_{ls}
                                              + \partial_{x_k} g(t+s,X^{(0,x)}_s)
                                         D^{ij}_{ks}\right\}
                                       ds\right] 
  \end{align*}
  Furthermore, by the Markov property, for $h\in (t,T)$, 
  $$
  u(t,x) = \mathbb E\left[-\int_t^h g(s,X^{(t,x)}_s) dt + u(h, X^{(t,x)}_h)\right],
  $$
  and an application of the It\^o formula yields:
  $$
  \mathbb E\int_t^h \left\{-g(s,X_s) +
    \partial_t(s,X_s)+\partial_{x_i} u(s,X_s) \mu^i(s,X_s) +
    \frac{1}{2}\partial^2_{x_i x_j}u(s,X_s)
    [\sigma^\top \sigma]_{ij}(s,X_s) \right\} dt=0,
  $$
  where we removed the superscript $(t,x)$ to save space. Dividing
  both sides by $h-t$ and passing to the limit $h\to t$, we get
  \eqref{PDE}. 
\end{itemize}
\end{remark}

\begin{lemma}\label{apriori.lm} Let Assumptions (X-SDE) and
  (X-PDE) be satisfied. Fix $m^*_0 \in \mathcal P({\mathcal O})$.
\begin{itemize}
\item[i.]\textcolor{black}{Let $g:\mathcal O \mapsto \mathbb{R}^+$} be a continuous  function with
polynomial growth. Then
  almost everywhere on $t\in[0,T]$, and for $m \in \mathcal A(m^*_0)$,
$$
\int_{\mathcal O} g(x)  m_t(dx) \leq \int_{{\mathcal O}} m^*_0(dx)
\mathbb E[g(X^{x}_{t})\mathbf 1_{t< \tau^{x}_{\mathcal O}}].
$$
\item[ii.] \textcolor{black}{Let $g\in C^2(\mathcal O; \mathbb{R})$ such that $g$, $|\!|\nabla_{_X} g|\!|$ and
  $|\!|H_{_X} g|\!|$ are bounded.} Then, for $m \in \mathcal A(m^*_0)$ and for
  every $\psi \in C^1([0,T])$,
$$
{\int_0^T \psi'(t)\left(\int_{\mathcal O} g(x) m_t(dx)\right)dt \leq C
\|\psi\|_\infty},
$$
for some $C>0$.
\end{itemize}
\end{lemma}
\begin{proof}
Part i. Assume that $f$ and $g$ are $C^\infty$ bounded positive
functions {\color{black}with bounded derivatives of all orders}, and let $u$ be the solution of 
$$
 \left(\frac{\partial u}{\partial t} + \mathcal Lu\right)(t,x)  = - g(x) f(t),
$$
described in Assumption (X-PDE).
By It\^o's formula, for $x\in \mathcal O$,
$$
\textcolor{black}
{-u(t,x) = \int_t^{T\wedge \tau^{(t,x)}_{\mathcal O}} \left\{\frac{\partial u}{\partial t}(s, X^{(t,x)}_s) + \mathcal
  Lu(s, X^{(t,x)}_s)\right\}ds + \int_t^{T\wedge \tau^{(t,x)}_{\mathcal O}} (\nabla_X u)^\top (s, X^{(t,x)}_s) \sigma (s, X^{(t,x)}_s) dW_s.}
$$
Taking the expectation and \textcolor{black}{using the equation satisfied by $u$, the
fact that $|\!|\nabla_X u |\!|$ has polynomial growth and the a priori estimates on the strong solution of the SDE (i.e. $\sup_t \mathbb{E}[|\!|X_t|\!|^p]<\infty$, for all $p \geq 1$), we get}

$$
u(t,x) = \mathbb E\left[\int_t^{T\wedge \tau^{(t,x)}_{\mathcal O}}g(X^{(t,x)}_s) f(s) ds\right],
$$
which means that $u$ is an admissible test function \textcolor{black}{in the sense of Definition \ref{DEF}}. Substituting the
above expression for $u$ \textcolor{black}{into
the constraint \eqref{constraint}}, we have
$$
\int_0^T {f(t)}\int_{{\mathcal O}} \mathbb
E\left[g(X^{x}_{t})\mathbf 1_{t< \tau^{x}_{\mathcal O}}\right] m^*_0(dx) dt\geq \int_0^T f(t)\int_{\mathcal O} g(x)
m_t(dx)dt.
$$
Since $f$ is arbitrary, this implies that
$$
\int_{{\mathcal O}} m^*_0(dx) \mathbb
E\left[g(X^{x}_{t}) \mathbf 1_{t< \tau^{x}_{\mathcal O}}\right]\geq \int_{\mathcal O} g(x)
m_t(dx),
$$
$t$-almost everywhere on $[0,T]$. The result may be
extended to a positive continuous function $g$ with polynomial growth by considering a sequence of functions $g^{l,n,m}(x):=g^l(x) \phi^{n,m}(x)$, where  $g^l:=\rho^l \star g$ converges uniformly on compact sets to $g$ (see Prop. 4.21 in \cite{B2010}), $\phi^{n,m}:=\rho^m \star \psi^n$ converges pointwise to $\psi^n$, where $(\rho^l)_{l \geq 1}$, $(\rho^m)_{m \geq 1}$ are two sequences of mollifiers  and $\psi^n(x):=\textbf{1}_{x \in K^n}$, with $K^n$  a sequence of increasing compact sets approximating the open set $\mathcal{O}$ (exhaustion by compact sets of the set $\mathcal{O}$). Note that all elements of the sequence of functions $(g^{l,n,m})_{l,n,m}$ admit bounded derivatives of all orders (since they are continuous and have compact support). The result follows by applying first Lebesgue's Theorem, when taking the limit with respect to $l$ and $m$ and then the monotone convergence theorem when letting $n \rightarrow \infty$.\\

Part ii. \textcolor{black}{First remark that
$$
\int_{\mathcal O} g(x) m_t(dx),
$$
is bounded on $[0,T]$.} This implies that it is enough to prove
the result for $\psi \in C^{\infty}([0,T])$, because for $\psi \in
C^1([0,T])$, the derivative $\psi'$ may be approximated by smooth
functions in the uniform norm.

\textcolor{black}{By It\^o formula, for $s\leq \tau^{(t,x)}_{\mathcal O}$,
\begin{align*}
g(X^{(t,x)}_s) &= g(x) + \int_t^s (\nabla_X g)^\top(X^{(t,x)}_r)
dX^{(t,x)}_r + \frac{1}{2} \int_t^s  Tr(\sigma(r,X^{(t,x)}_r)^{\top} H_X g(X^{(t,x)}_r)\sigma(r,X^{(t,x)}_r))dr.
\end{align*}}
\textcolor{black}{Taking the expectation and integrating by parts we obtain
\begin{align*}
&\mathbb E\left[\int_t^{T\wedge \tau^{(t,x)}_{\mathcal O}} g(X^{(t,x)}_s) \psi'(s) ds\right]\\
&= \mathbb E\left[\int_t^{T\wedge \tau^{(t,x)}_{\mathcal O}} g(x)
  \psi'(s)ds\right]+ \mathbb E\left[\int_t^{T\wedge
  \tau^{(t,x)}_{\mathcal O}} \psi'(s) \int_t^s  (\nabla_X g)^\top(X^{(t,x)}_r)
  \mu(r,X^{(t,x)}_r) dr ds\right] \\ &\qquad +
  \textcolor{black}{\frac{1}{2}\,\mathbb E\left[\int_t^{T\wedge
  \tau^{(t,x)}_{\mathcal O}} \psi'(s) \int_t^s  Tr(\sigma(r,X^{(t,x)}_r)^{\top} H_X g(X^{(t,x)}_r)\sigma(r,X^{(t,x)}_r))dr ds\right]}\\
&  = g(x)\mathbb E[\psi({T\wedge \tau^{(t,x)}_{\mathcal O}})-\psi(t)] + \mathbb E\left[\int_t^{T\wedge \tau^{(t,x)}_{\mathcal O}}(\psi({T\wedge \tau^{(t,x)}_{\mathcal O}}) -\psi(r))
   (\nabla_X g)^\top(X^{(t,x)}_r) \mu(r,X^{(t,x)}_r) dr\right] \\&\qquad + \frac{1}{2}\textcolor{black}{\mathbb E\left[\int_t^{T\wedge \tau^{(t,x)}_{\mathcal O}}(\psi({T\wedge \tau^{(t,x)}_{\mathcal O}}) -\psi(r))
  Tr(\sigma(r,X^{(t,x)}_r)^{\top} H_X g(X^{(t,x)}_r)\sigma(r,X^{(t,x)}_r)) dr\right]\leq
                                                   C \|\psi\|_\infty},
\end{align*}}
for some constant $C<\infty$, \textcolor{black}{due to the bounds on $g$, $|\!|\nabla_X g|\!|$, $|\!|H_X g|\!|$, $|\!|\mu|\!|$ and $|\!|\sigma|\!|$.}
Then we can define the function
$$
u(t,x) = \mathbb E\left[\int_t^{T\wedge \tau^{(t,x)}_{\mathcal O}} g(X^{(t,x)}_s) \psi'(s) ds\right]  +
C\|\psi\|_\infty,
$$
which is an admissible test function by the same argument as the one
used in the first part.
This proves that
$$
\int_{\mathcal O} u(0,x) m^*_0(dx)\geq \int_0^T\psi'(t) \left(\int_{\mathcal O} g(x)
m_t(dx)\right)dt,
$$
and {since $u(0,x) \leq 2C\|\psi\|_\infty$} for all $x \in \mathcal{O}$, we get the statement
of the lemma.
\end{proof}
\begin{corollary} \label{bounded.cor}Under the assumptions of Lemma
  \ref{apriori.lm}, let $m^*_0 \in \mathcal P({\mathcal O})$, and let
  $\bar m_t(dx)$ be the distribution of the process $X$ started with
  initial distribution $m^*_0$ and killed at the first exit time from
  $\mathcal O$. Then for every $m\in \mathcal
  A(m^*_0)$, $m_t\leq \bar m_t$, \textcolor{black}{$dt$-almost everywhere}  on $[0,T]$. In particular,
  if $\bar m_t$ has a density then $m_t$ does as well.
\end{corollary}

\begin{proof}
Approximating the indicator function with a sequence of continuous
bounded functions and using the dominated convergence theorem, the
first part of the above lemma yields for all $a,b\in {\mathcal O}$ with
$a<b$ (where the inequality is interpreted componentwise),
$$
m_t([a,b])\leq \int_{[a,b]} \int_{\mathcal O} m^*_0(dx) \bar p^{x}(t,dz) =
\bar m_t([a,b]),
$$
where $\bar p^{x}(t,dz)$ is the transition distribution of the process
$X$ killed at $\tau^{x}_{\mathcal O}$.
\end{proof}

 In the following lemma we continue
the study of the properties of the set $\mathcal A(m^*_0)$. {\color{black}The
compactness of this set is established under the following
assumption. 
\begin{assumption}[\textbf{$m^*_0$-Compact}]
The initial distribution $m^*_0\in \mathcal P({\mathcal O})$ satisfies
\begin{align}
{\color{black}\int_{\mathcal O} \ln\{1+|\!|x|\!|\}m^*_0(dx) <\infty. }\label{compact.eq}
\end{align}
 \end{assumption}

\begin{lemma}\label{compact.lm} Let Assumptions (X-SDE) and
  ($m^*_0$-Compact) be satisfied.
Then the set $\mathcal A(m^*_0)$ is {sequentially compact}.
\end{lemma}}
\begin{proof}
Let us first show the tightness of the associated set of measures on $[0,T]\times
{\mathcal O}$. For $i=1,\dots,d$, define the function
$$
\textcolor{black}{\phi^i_A(x) = \ln \left\{1+|x_i|^3\left(\frac{3x_i^2}{5A^2} - \frac{3|x_i|}{2A} +
    1\right)\right\}\mathbf 1_{|x_i|\leq A} +
\ln\left\{1+\frac{A^3}{10}\right\}\mathbf 1_{|x_i|>A}},
$$
with $A\ge 0$.
Remark that
$$
\textcolor{black}{\partial_{x_i}\phi^i_A(x) = \frac{3x_i^2(1-|x_i|/A)^2}{1+|x_i|^3\left(\frac{3x_i^2}{5A^2} - \frac{3|x_i|}{2A} +
    1\right)}\mathrm{sgn}\,(x_i)\mathbf 1_{|x_i|\leq A}},
$$
\textcolor{black}{and $\partial_{x_j} \phi^i_A(x)=0$, for all $j\neq i$,
from which it is easy to see that $\phi^i_A$ is twice continuously
differentiable on its entire domain, and that the expressions
$\nabla_{_X}\phi^i_A(x)$, $x^T\nabla_{_X}\phi^i_A(x)$, $H_{_X}\phi^i_A(x)$ and $x^\top H_{_X}\phi^i_A(x)
x$ are bounded on ${\mathcal O}$ by a constant
independent from $A$.  In addition, as $A\to \infty$,
$\phi^i_A(x)$ converges in a monotone fashion to the limiting function
$\phi^{i*}(x) = \ln\{1+|x_i|^3\}$.
Now, consider the
test function $u_A(t,x) = (T-t)\sum_{i=1}^d\phi^i_A(x)$.
It follows that
\begin{multline*}
T\int_{\mathcal O} \phi_A(x) m_0^*(dx)-\int_0^T
\int_{\mathcal O}\phi_A(x) m_t(dx)\,dt \\+ \int_0^T
\int_{\mathcal O} (T-t)\sum_{i=1}^d\left\{ \mu_i(t,x)
  \partial_{x_i}\phi^i_A(x) + \frac{\|\sigma_i(t,x)\|^2}{2}
  \partial^2_{x_i x_i}\phi^i_A(x) \right\}m_t(dx)\,dt\geq 0,
\end{multline*}
for $m \in \mathcal A(m^*_0)$.}
From the boundedness of $\mu$ and $\sigma$ and the above
observations, we deduce that the expression within the brackets in the
last term is bounded uniformly on $A$. The limits of the first two
terms, on the other hand, are computed by monotone
convergence. {\color{black}Letting $\phi^*(x):= \sum_{i=1}^d \phi^{i*}(x)$},  we
conclude that there exists a constant $C<\infty$ such that
$$
\int_0^T
\int_{\mathcal O}\phi^*(x) \mu(dt,dx) \leq C + T\int_{\mathcal O} \phi^*(x) m_0^*(dx),
$$
from which the tightness follows \footnote{\textcolor{black}{For sake of clarity, we precise the tightness criteria. Let $F$ be a topological space equipped with its Borel sigma-field. Let $(\mu_{i})_{i \in I}$ be a {\color{black}flow of measures} on $(F, \mathcal{B}(F))$. If there exists a measurable function $\phi:F\mapsto [0, \infty]$ with compact level sets such that $C:=\sup_{i \in I} \int_F \phi(x)d\mu_i(x)<\infty$, then $(\mu_i)_{i \in I}$ is tight (the proof follows immediately by the measure version of the Markov inequality).}}. Moreover, taking $g=1$ in Lemma
\ref{apriori.lm} we see that $\mathcal A(m^*_0)$ is uniformly bounded.
Therefore, by Prokhorov's theorem (Theorem 8.6.2 in \cite{bogachev2007measure}), from any sequence of {\color{black}flows of measures} $(m^n)_{n\geq 1} \subseteq \mathcal A(m^*_0)$, one can
 extract a subsequence, also denoted by $(m^n)_{n\geq 1} $,
such that the sequence of associated measures on $[0,T]\times {\mathcal O}$,
$(\mu^n)_{n\geq 1}$ converges weakly to a limiting measure
$\mu^*$. By weak convergence, the measure $\mu^*$ also satisfies the
constraints of $\mathcal A(m^*_0)$ i.e., for every test function $u$,
$$
\int_{\mathcal O} u(0,x) m^*_0(dx) + \int_0^T \int_{\mathcal O}
\left\{\frac{\partial u}{\partial t} + \mathcal L
  u\right\}\mu^*(dt,dx) \geq 0.
$$
Taking the test function $u(t,x) = \int_t^T f(s)ds$ with $f$ a
positive continuous function, we have
$$
\int_0^T f(t) dt \int_{\mathcal O} m^*_0(dx) \geq \int_0^T \int _{\mathcal O} f(t)
\mu^*(dt,dx).
$$
We conclude that $\mu^*$
is a bounded measure and the measure $\int_{\mathcal O} \mu^*(dt,dx)$ on
$[0,T]$ is
absolutely continuous with respect to the Lebesgue measure, which
means that we can write $\mu^*(dt,dx) = m^*_t(dx)\, dt$ for some
$m^*\in \mathcal A(m^*_0)$. {\color{black}The positivity of the limiting measure
flow follows from weak convergence and absolute continuity.}
\end{proof}



The following proposition is an existence result for the relaxed
optimal stopping problem. We need the following assumption on $f$.
{\color{black}\begin{assumption}[\textbf{$f$-Exist}]
 One of the following alternative
conditions holds true:
\begin{itemize}
\item[i.] The mapping $(t,x)\mapsto f(t,x)$ is continuous on $[0,T]\times
{\mathcal O}$ and satisfies
\begin{align*}
&\int_0^T \int_{{\mathcal O}} |f(t,x)| \bar m_t(dx)\, dt<\infty,
\end{align*}
where $\bar m_t$ is the distribution at time $t$ of the process $X$
started with initial distribution $m_0^*$.
\item[ii.] The function $f$ is of the form
$$
f(t,x) = \sum_{i=1}^n\bar f_i(t) g_i(x),
$$
where $n\geq 1$ and for each $i$, $g_i\in C^2(\mathcal O; \mathbb{R})$
is such that $g_i$, $|\!|\nabla_{_X} g_i|\!|$ and
  $|\!|H_{_X} g_i|\!|$ are bounded.,  and $\bar f_i$ is bounded measurable.
\end{itemize}
 \end{assumption}}

\begin{proposition}\label{stoppingexist.lm}Let Assumptions
  (X-SDE),  (X-PDE), ($m^*_0$-Compact) and ($f$-Exist) be satisfied.
Then there
exists $m^*\in \mathcal A(m^*_0)$ which maximizes the functional
$$
 m \mapsto \int_0^T\int_{\mathcal O}  f(t,x) {m}_t (dx)\, dt,
$$
 over
all $ m \in \mathcal A(m^*_0)$.
\end{proposition}
\begin{proof}
Choose a maximizing sequence of {\color{black}flows of measures} $(m^n)_{n\geq 1}{\subseteq} \mathcal A(m^*_0)$. By Lemma
\ref{compact.lm}, it has a subsequence, also denoted by $(m^n)_{n\geq 1}$, which
converges weakly to a limit $m^* \in \mathcal A(m^*_0)$.
To show that $m^*$ is a maximizer of \eqref{relcost}, we consider
separately the two alternative assumptions of the proposition.

Suppose that Assumption i.~holds true. Fix $\varepsilon>0$. By the continuity of $f$ and the
integrability assumption, there exists $0\le M<\infty$ such that
$$
\int_{[0,T]\times {\mathcal O}, f(t,x)>M} f(t,x) \bar m_t(dx)\, dt +
\int_{[0,T]\times {\mathcal O}, f(t,x)<-M} |f(t,x)| \bar m_t(dx)\, dt
<\varepsilon.
$$
Then, by weak convergence and by Corollary \ref{bounded.cor},
\begin{align}\label{ineqq}
&\limsup_n \int_0^T \int_{{\mathcal O}} f(t,x) m^n_t (dx)\, dt \leq
\varepsilon + \limsup_n \int_0^T\int_{\mathcal O} (M\wedge
f(t,x))\vee (-M) m^n_t(dx) dt \nonumber \\
& = \varepsilon + \int_0^T\int_{\mathcal O} (M\wedge
f(t,x))\vee (-M) m^*_t(dx) dt \nonumber \\
&\leq 2\varepsilon + \int_0^T \int_{\mathcal O} f(t,x) m^*_t(dx) dt.
\end{align}
Since $\varepsilon$ is arbitrary, $m^n$ is a maximizing sequence and $m^* \in \mathcal A(m^*_0)$,
this finishes the proof.

Suppose now that Assumption ii.~holds true instead. Without loss of
generality it is enough to consider the case where $n=1$, and we omit
the index $i$. 
Consider the mapping $G^n: [0,T]\mapsto \mathbb R$ defined by $G^n(t)
= \int_{{\mathcal O}}
g(x)m^n_t (dx)$.
By Lemma \ref{apriori.lm} and Proposition
3.6 in \cite{ambrosio2000functions}, $G^n$ is then of bounded
variation on $[0,T]$. Then,
by Theorem 3.23 in the above reference, up to taking a subsequence, we
may assume that the sequence of mappings $(G^n)_{n\geq 1}$ converges
in $L^1([0,T])$ to some mapping $G^*$. On the other hand, in view of the
weak convergence, for any continuous function $f:[0,T]\mapsto \mathbb
R$,
$$
\int_0^T f(t) G^n(t) dt \to \int_0^T f(t)  \int_{{\mathcal O}} g(x) m^*_t (dx) dt.
$$
This shows that $G^*(t)
= \int_{{\mathcal O}} g(x) m^*_t (dx)$. We conclude that
$$
\int_0^T \bar f(t) \int_{{\mathcal O}}
g(x) m^n_t (dx)dt\to \int_0^T \bar f(t)\int_{{\mathcal O}}
g(x) m^*_t (dx)dt,
$$
as $n\to \infty$. 
\end{proof}

\section{Relaxed formulation of the optimal stopping MFG problem}
We now give the definition of Nash equilibrium for the relaxed MFG
optimal stopping problem. {\color{black}For the problem to be well-defined, we
impose the following minimal assumption on the reward function $f$:
\color{black}\begin{assumption}[\textbf{$f$-min-MFG}]
For every $m\in \mathcal A(m^*_0)$, the map 
$$
(t,x) \mapsto f(t,x,m_t)
$$
is measurable and satisfies
$$
\int_0^T \int_{\mathcal O} (f(t,x,m_t))_- \bar m_t(dx) <\infty. 
$$
\end{assumption}}
\begin{definition}
Given the initial distribution $m^*_0$, a {\color{black}flow of
  measures} $m^* \in \mathcal{A}(m^*_0)$ is a Nash equilibrium for the
relaxed MFG optimal stopping problem (or ``relaxed Nash equilibrium'')
if
{\color{black}$$
\int_0^T \int_{\mathcal O}  {f}(t,x,m^*_t)m^*_t(dx)\,dt <\infty
$$
and}
\begin{align*}
\int_0^T \int_{\mathcal O} {f}(t,x,m^*_t )m_t(dx)\,dt \leq \int_0^T \int_{\mathcal O}  {f}(t,x,m^*_t)m^*_t(dx)\,dt,
\end{align*}
for all $m \in \mathcal{A}(m^*_0)$.
\end{definition}
In other words, the set of Nash equilibria coincides with the set of
fixed points of the \textcolor{black}{set-valued mapping ${\Theta: \mathcal A(m_0^*) \to
 2^{\mathcal{A}(m_0^*)}, {\rm \,\, with\,\, } 2^{\mathcal{A}(m_0^*)}  {\rm \,\, the\, family\, of }}$ ${ {\rm sets \, over \,\, }\mathcal{A}(m_0^*)}$}, defined by
$$
\Theta (m) =  \argmax_{\hat m \in \mathcal A(m^*_0)}
\int_0^T \int_{{\mathcal O}} f(t,x,m_t) \hat m_t (dx)\,dt,
$$
which is well defined whenever the function $(t,x)\mapsto f(t,x,m_t)$
satisfies the conditions of Proposition \ref{stoppingexist.lm}.

{\color{black}The next theorem estalishes existence of the MFG equilibrium under the
following assumption.
\begin{assumption}[\textbf{$f$-Exist-MFG}]
Let the reward function $f$ be of the form
$$
f(t,x,m) = \sum_{i=1}^K \bar f_i \left(t,\int_{\mathcal O} \bar g_i(x) m_t(dx)\right) g_i(x),
$$
where, for each $i$,  \textcolor{black}{$g_i, \bar{g}_i \in C^2(\mathcal O; \mathbb{R})$ 
are such that $g_i, \bar{g}_i$, $|\!|\nabla_{_X} g_i|\!|$, $|\!|\nabla_{_X} \bar{g}_i|\!|$,
  $|\!|H_{_X} g_i|\!|$,  $|\!|H_{_X} \bar{g}_i|\!|$ are bounded},  and $\bar f_i$ is bounded measurable and continuous with respect to its second argument.
\end{assumption}}
 \begin{theorem}\label{exist} Let Assumptions (X-SDE), (X-PDE),
  ($m^*_0$-Compact) and ($f$-Exist-MFG) be satisfied. Then there exists a Nash equilibrium for the relaxed MFG problem.
\end{theorem}
\begin{proof}
We shall use the Fan-Glicksberg fixed{-}point theorem (\textcolor{black}{Theorem 7.1 in}
\cite{mclennan2018advanced}). We have seen that $V$ is a locally
convex space; moreover, the subset $\mathcal A(m_0^*) \subseteq V$ is
compact (by Lemma \ref{compact.lm} and \textcolor{black}{since $\mathcal{A}(m_0^*)$ is included in the space of
positive and finite measures on a separable metric space, which is metrizable}), \textcolor{black}{convex} and nonempty. The mapping
$\Theta$ is clearly convex. Therefore, to prove that it has a fixed point it suffices
to check that it is upper semicontinuous. In other words, we check that it has a
closed graph (see Proposition 5.1.3 in \cite{mclennan2018advanced}), where the graph is defined by
$$
{\text{Gr}(\Theta) = \{(m,\bar m)\in (\mathcal A(m^*_0))^2: m\in \Theta(\bar m)\}}.
$$
To show that ${\text{Gr} (\Theta)}$ is closed it suffices to check that for any
two sequences $(m^n)_{n\ge 1}\subseteq \mathcal A(m^*_0)$ and $(\bar m^n)_{n\ge 1}\subseteq
\mathcal A(m^*_0)$ which converge weakly to $m\in \mathcal A(m^*_0)$ and
$\bar m\in \mathcal A(m^*_0)$ respectively, and such that
$$
\int_0^T \int_{{\mathcal O}} f(t,x,\bar m^n_t)  m^n_t (dx)\,dt \geq \int_0^T
\int_{{\mathcal O}} f(t,x,\bar m^n_t)  \hat m_t (dx)\,dt,
$$
for every $\hat m\in \mathcal A(m^*_0)$,
we have
$$
\int_0^T \int_{{\mathcal O}} f(t,x,\bar m_t)  m_t (dx)\,dt \geq \int_0^T
\int_{{\mathcal O}} f(t,x,\bar m_t)  \hat m_t (dx)\,dt,
$$
for every $\hat m\in \mathcal A(m^*_0)$.
To prove this, it is enough to show that, up to taking a subsequence,
\begin{align}\label{conv1}
\int_0^T \int_{{\mathcal O}} f(t,x,\bar m_t)  m_t (dx)\,dt = \lim_{n}\int_0^T \int_{{\mathcal O}} f(t,x,\bar m^n_t)  m^n_t (dx)\,dt,
\end{align}
{and
\begin{align}\label{conv2}
\int_0^T \int_{{\mathcal O}} f(t,x,\bar m_t)  \hat{m}_t (dx)\,dt = \lim_{n}\int_0^T\int_{{\mathcal O}} f(t,x,\bar m^n_t)  \hat m_t (dx)\,dt.
\end{align}
We will only show that $\eqref{conv1}$ holds true, since the convergence given by \eqref{conv2} follows by the same arguments.} It is enough to
consider the case $K=1$ and we drop the index $i$. We therefore need
to prove
\begin{align}
\int_0^T  \bar f(t,\bar g * \bar m_t) \, g * m_t\,dt = \lim_n \int_0^T
  \bar f(t,\bar g * \bar m^n_t) \, g* m^n_t \,dt,\label{limprop2}
\end{align}
where we write $g*m$ as a shorthand for $\int_{\mathcal O} g(x) m(dx)$.
As in the proof of
Proposition \ref{stoppingexist.lm}, we may show that $\bar g * \bar m^{n}$ converges to $\bar g*\bar m$ in
$L^1([0,T])$.  Similarly, we may show that $g * m^{n}$ converges to $ g*m$ in
$L^1([0,T])$. Since $f$ is continuous, $\bar f(t,\bar g * \bar m^n_t) \, g*
m^n_t $ converges almost everywhere to $\bar f(t,\bar g * \bar m_t) \, g*
m_t $. Further, by Corollary \ref{bounded.cor}, $g*
m^n_t $ is uniformly bounded, and \eqref{limprop2} follows from the
dominated convergence theorem.
\end{proof}

\paragraph{Uniqueness of the Nash value for the relaxed MFG problem}
We prove here the uniqueness result of the Nash equilibrium value for the relaxed problem, which holds under the following  assumption on the map $f$.
{\color{black}\begin{assumption}[\textbf{$f$-Uniq-MFG}]
The function $f$ takes the following form
\begin{equation*}
f(t,x,m)=g(x)\bar f\left (t,\int_{\mathcal O} g(x)m_t(dx)\right) + h(t,x),
\end{equation*}
where $g\in C^2(\mathcal O; \mathbb{R})$
is such that $g$, $|\!|\nabla_{_X} g|\!|$ and
  $|\!|H_{_X} g|\!|$ are bounded.,  $h:[0,T]\times\mathcal O \mapsto \mathbb{R}$ is continuous, with polynomial growth in $x$ and $\bar f: [0,T]\times \mathbb{R} \mapsto \mathbb{R}$ is bounded measurable, continuous {and decreasing}  in the second
argument.
\end{assumption}}
\begin{remark}\label{antim}
Note that under Assumption ($f$-Uniq-MFG), the function $f$
satisfies for each $t$ and all $m^1 \in \mathcal A(m^*_0)$ and $m^2
\in \mathcal A(m^*_0)$ the following antimonotonicity condition
\begin{align}
&\int_{\mathcal O}
  \left(f(t,x,m^1_t)-f(t,x,m^2_t)\right)\left(m^1_t(dx)-m^2_t(dx)\right) \notag\\
& = (\bar f(t,g*m^1_t) - \bar f(t,g*m^2_t)) (g*m^1_t-g*m^2_t)\leq 0. \notag
\end{align}
\end{remark}

\begin{theorem}[\bf{Uniqueness of the Nash value}]\label{Uniqueness}
Let $m^*$ and $\bar{m}$ be two Nash equilibria for the relaxed
problem and let Assumption ($f$-Uniq-MFG) be satisfied. Then,
$$
\bar f(t,g*m^*_t) = \bar f(t,g*\bar m_t),
$$
almost everywhere on $[0,T]$,
and in particular they lead to the same value of the relaxed fixed point problem, that is $\int_0^T \int_{\mathcal O} f(t,x,m^*)m^*_t(dx)=\int_0^T \int_{\mathcal O} f(t,x,\bar{m})\bar{m}_t(dx)$.
\end{theorem}
\begin{proof}
Since $m^*$ is a Nash equilibrium, we get that
\begin{align*}
\int_0^T \int_{\mathcal O} f(t,x,m^*)\bar{m}_t(dx)\leq \int_0^T \int_{\mathcal O} f(t,x,m^*)m_t^*(dx)dt.
\end{align*}
Since $\bar{m}$ is also a Nash equilibrium, we obtain
\begin{align*}
\int_0^T \int_{\mathcal O} f(t,x,\bar{m})m_t^*(dx)dt\leq \int_0^T \int_{\mathcal O} f(t,x,\bar{m})\bar{m}_t(dx)dt.
\end{align*}
From the two above inequalities, we derive that
\begin{align*}
\int_0^T \int_{\mathcal O} (f(t,x,\bar{m})-f(t,x,m^*))(\bar{m}_t(dx)-m_t^*(dx))dt \geq 0.
\end{align*}
The antimonotonicity property of the map ${f}$ then implies that
\begin{align*}
\int_{\mathcal O} (f(t,x,\bar{m})-f(t,x,m^*))(\bar{m}_t(dx)-m^*_t(dx)) = 0,
\end{align*}
almost everywhere on $[0,T]$, or in other words that
$$
(\bar f(t,g*m^1_t) - \bar f(t,g*m^2_t)) (g*m^1_t-g*m^2_t)=0,
$$
almost everywhere on $[0,T]$, which implies that
$$
\bar f(t,g*m^1_t) = \bar f(t,g*m^2_t),
$$
almost everywhere on $[0,T]$. Integrating over $[0,T]\times {\mathcal O}$ we
see that the two equilibria lead to the same value.
\end{proof}
\textcolor{black}{
\begin{remark}
A natural question is to see if one can use a relaxed Nash equilibrium corresponding to the MFG game problem in order to construct a $\varepsilon$-Nash equilibria for the $N$-player game. A possible way to do it, is to show that, given $m^*$ a relaxed MFG equilibrium, then the empirical measures 
\begin{align}\label{empiricalmeas2}
m^N_t(dx) = \frac{1}{N} \sum_{i=1}^N \delta_{X^i_t}(dx) \mathbf 1_{t
\leq \textcolor{black}{\tau^{i}} \wedge \tau_{\mathcal O}^i},
\end{align}
correspond to a $\varepsilon$-Nash equilibria, 
where, for every $i \in \{1,2,...,N \}$, $\tau_{i}$ maximizes
$$
\max_{\tau }\mathbb E\left[\int_0^{\tau\wedge \tau_{\mathcal O}^i} e^{-\rho t} \tilde f(t,X^i_t, m_t^*)
dt \right].
$$
This problem is left for further research.
\end{remark}}

\section{Relation between the relaxed and the strong formulation of the single-agent optimal stopping and of the MFG problem \textcolor{black}{and relation with mixed solutions}}\label{relation.sec}
In this section we provide the relation between the relaxed and the strong formulation of the single-agent optimal stopping problem and of the MFG problem, \textcolor{black}{as well as with the mixed solutions introduced in \cite{bertucci2017optimal}}. We make here the following additional assumption.
\begin{assumption}[\textbf{X-Reg}]${}$
\begin{enumerate}[i.]
\item \textcolor{black}{The domain ${\mathcal O}$ is an open bounded domain  of
$\mathbb{R}^d$, with boundary $\Gamma:=\partial \mathcal{O}$ of class
$C^2$ and  the process $X_\cdot$ {\color{black}started with initial distribution
$m^*_0$ and killed at the first exit time of $\mathcal O$} has a distribution
$\bar m_t$, which, for each $t$, has a square integrable density with respect to the Lebesgue measure.
\item  $\sigma$ satisfies the uniform ellipticity condition.}
\end{enumerate}
\end{assumption}
\begin{remark}{\color{black}
  Let $\mathcal O$ be as in Assumption (X-Reg), assume that
  $\sigma$ satisfies the uniform ellipticity condition, that the
  coefficients $a =
  \sigma^\top \sigma$ and $\mu$ are uniformly Lipschitz continuous on
  $[0,T]\times \mathcal O$ and that the initial distribution $m^*_0$
  admits a bounded density with respect to the Lebesgue measure. Then, by Theorem 3.16 in
  \cite{friedman83}, the operator $\mathcal L$ admits a
  Green function $G(x,t; \xi, T)$, which is continuous in $\xi$ for
  all $T>t$. Moreover, the Green function admits an Aronson-type
  estimate of the form
  \begin{align}
  G(x,t; \xi, T) \leq c(T-t)^{-d/2} \exp\left(-C \frac{\|x-\xi\|^2}{T-t}\right),\label{aronson}
  \end{align}
  see Equation (16.16) in \cite{ladyzh}. 
  This means that the solution $u$ to the equation
  $$
  \frac{\partial u}{\partial t} + \mathcal Lu =0
  $$
  with boundary condition $u|_{\partial \mathcal O} = 0$ and terminal
  condition $u(T,\xi) = \phi(\xi)$ is given by
  $$
  u(t,x) = \int_{\mathcal O} G(x,t; \xi, T) \phi(\xi)d\xi. 
  $$
  On the other hand, by Theorem 5.2 in \cite{F75}, this solution
  is given by
  $$
  u(t,x) = \mathbb E[\phi(X^{(t,x)}_T)\mathbf 1_{\tau^x_{\mathcal
      O}>T}]. 
  $$
  We conclude that the Green function coincides with the density of
  the process started at $(t,x)$ and killed at the first exist time
  from $\mathcal O$. The density of the process started with the
  initial distribution $m_0^*$ is therefore given by
  $$
  \bar m_t(\xi) = \int_{\mathcal O} m^*_0(x) G(x,0; \xi, T)dx.
  $$
  Since $m^*_0$ is bounded by assumption, we conclude using the bound
  \eqref{aronson} that the density \textcolor{black}{$\bar{m}_t(\xi)$} is uniformly bounded on
  $[0,T]$. 
   Note that the process satisfying the conditions
  given in this remark also satisfies the assumptions (X-SDE)
  and (X-PDE) (see Remark \ref{pde.rem}). 
}
  \end{remark}
Note that, by Corollary \ref{bounded.cor}, we derive that $m_t$ admits a  square integrable density with respect to the Lebesgue measure, for each $m \in \mathcal{A}(m_0^*)$ and for a.e. $t \in (0,T]$.

Let \textcolor{black}{$W$ be a standard $K$-dimensional Brownian motion} and $X_0$ be a random variable
with distribution $m^*_0$, independent from $W$. {We suppose that
  $X_0$ is valued in $\mathcal{O}$ and that $m^*_0$ admits a square
  integrable density with respect to the Lebesgue measure.} In the sequel, we
denote by $\mathbb{F}$ the filtration given by $\mathcal F_t =
\sigma(W_s, 0\leq s\leq t) \vee\sigma(X_0)\vee \mathcal{N}$, where $\mathcal{N}$ denotes
the sets of zero measure.  Moreover, $\mathcal{T}([t,T])$ denotes the set of
stopping times with respect to this filtration with values in $[t,T]$. \textcolor{black}{We also denote by $\mathcal{T}^t_W([t,T])$ the set of stopping times with respect to the \textcolor{black}{(completed)} filtration generated by the translated Brownian motion $W^t_s:=W_s-W_t$, $s\geq t$,  with values in $[t,T]$.} 

\textcolor{black}{We address first the case of the single-agent optimal stopping problem.}

\begin{theorem}\textbf{[Single-Agent optimal stopping problem]}\label{TH}
\textcolor{black}{Let Assumptions (X-SDE), (X-PDE), (X-Reg) and
  \textcolor{black}{($f$-Exist)(ii)} be satisfied.} Let $v$ be the value function of the following optimal stopping problem
\begin{align}\label{valuefct}
v(t,x)=\sup_{\textcolor{black}{\tau \in {\mathcal{T}^{t}_W([t,T])}}} \mathbb{E}\left[\int_t^{\tau \wedge \tau_{\mathcal O}^{(t,x)}} f(s,X_s^{(t,x)})ds\right],
\end{align}
with $(t,x) \in [0,T] \times \mathbb{R}$ and $\tau_{{\mathcal O}}^{(t,x)}:=\inf \{ s \geq t: \,\,\, X_s^{(t,x)} \notin {\mathcal O} \}$.
We have
\begin{enumerate}[i.]
\item $\int_{\mathcal O} v(0,x)m^*_0(dx)=\underset{m \in \mathcal{A}(m^*_0)}{\sup}   \int_0^T \int_{\mathcal O} f(s,x)m_s(dx)ds.$
\item Let $x \in \mathcal{O}$ and define $\bar{\tau}^{x}:=\inf \{0 \leq s \leq T:\,\, v(s,X_s^{x})=0\}$. Then the measure $m^*$ given by  $m_t^*(A):=\int_{\mathcal O} m^*_0(dx) \mathbb{P}[X_t^{x} \in A,t <\bar{\tau}^{x}]$ for all $A \in \mathcal{B}({\mathcal O})$ is a maximizer of the map $m \in \mathcal{A}(m_0^*) \mapsto  \int_0^T \int_{\mathcal O} f(s,x)m_s(dx)ds.$
\item Let $\bar{m}$ be a maximizer of the map $m \in \mathcal{A}(m_0^*) \mapsto \int_0^T \int_{\mathcal O} f(s,x)m_s(dx)ds.$ Then it satisfies:
\begin{enumerate}[a.]
\item $\int_{\mathcal{S}}f(t,x)\bar{m}_t(dx)dt=0$, with $\mathcal{S}:=\{(t,x) \in [0,T] \times {\mathcal O}: \,\, v(t,x)=0\}.$
\item For all ${C_c^\infty}$ functions $\phi$ such that $\supp \phi \subseteq \{(t,x) \in [0,T] \times {\mathcal O}:\,\, v>0 \}$, the following holds
\begin{align}
\int_0^T\int_{{\mathcal O}}
\left\{\frac{\partial \phi}{\partial t} + \mathcal L
  \phi\right\}\bar{m}_t(dx)dt+\int_{\mathcal O} \phi(0,x) m^*_0(dx)=0.
\end{align}
\end{enumerate}
\end{enumerate}
\end{theorem}
\begin{proof}
\textcolor{black}{Part i.}  By \textcolor{black}{Theorem 4.7, Chapter 3, in \cite{bensoussan1982applications}}, the
value function defined by \eqref{valuefct} is \textcolor{black}{a solution}
belonging to $W^{2,1,2}(\mathcal{Q})$, with $\mathcal{Q}:=(0,T) \times \mathcal{O}$ \footnote{\textcolor{black}{The Sobolev space $ W^{2,1,2}(\mathcal{Q})$ represents the set of functions $u$ such that $\partial_{t} u, \partial_{x_i} u, \partial_{x_i x_j} u \in L^2(\mathcal{Q})$, with $i,j=\overline{1,d}$, where the derivatives are understood in the sense of distributions.}}, which satisfies the following variational
inequality
\begin{align} \label{ineq}
\min \bigg(-\frac{\partial v}{\partial t}(t,x)-\mathcal{L}v(t,x)-f(t,x), v(t,x)\bigg)=0, \,\,\, (t,x) \in (0,T) \times {\mathcal O},\,\, \nonumber \\
v(t,x)=0, \,\, t \in (0,T), \,\, x \in \partial {\mathcal O}, \nonumber \\
v(T,x)=0, \,\, x \in {\mathcal O}.
\end{align}
First note that, by Lemma \ref{repres}, we have
\begin{align}\label{valuefct11}
v(0,X_0)=\underset{\tau \in \mathcal{T}([0,T])}\esssup\,\,\, \mathbb{E}\left[\int_0^{\tau \wedge \tau_{\mathcal O}^{X_0}}  f(s,X_s^{X_0})ds|\mathcal{F}_0\right] {\rm \,\, a.s.\,\,}
\end{align}
By classical results on optimal stopping and associated reflected Backward SDEs with random terminal time $T \wedge \mathcal{\tau}_\mathcal{O}^{X_0}$ (see e.g. Proposition 2.3 in  \cite{el1997reflected}), we get
\begin{align}\label{stop}
v(0,X_0)=\mathbb{E}\left[\int_0^{\bar{\tau}^{X_0}}  f(s,X_s^{X_0})ds|\mathcal{F}_0\right] {\rm \,\, a.s.\,\,},
\end{align}
where
\begin{align}\label{opt}
\bar{\tau}^{X_0}:=\inf \{0 \leq t \leq T:\,\, v(t,X_t^{X_0})=0\}.
\end{align}
Note that, by definition of the value function $v$, we have $\bar{\tau}^{X_0} \leq \tau_{\mathcal O}^{X_0} \wedge T$ a.s.\\
Taking now the expectation in $\eqref{stop}$, we derive that
\begin{align*}
\mathbb{E}[v(0,X_0)]=\mathbb{E}\left[\int_0^{\bar{\tau}^{X_0}}  f(s,X_s^{X_0})ds\right].
\end{align*}
 Remark that the occupational measure associated with the diffusion process $X_\cdot$ killed at the stopping time $\bar{\tau}^{X_0}$, that is $m_t(A):=\int_{\mathcal O} m^*_0(dx) \mathbb{P}[X_t^{x} \in A, t<\bar{\tau}^{x}]$, belongs to $\mathcal{A}(m_0^*)$.
Therefore, we have
\begin{align*}
\int_{\mathcal O} v(0,x)m^*_0(dx) \leq \sup_{m \in \mathcal{A}(m^*_0)}  \int_0^T \int_{\mathcal O} f(s,x)m_s(dx)ds.
\end{align*}

We now show the converse inequality.
Fix $m \in \mathcal{A}(m_0^*)$. Using a classical method of regularisation by convolution with a standard mollifier, with respect to both time and space (see, e.g., an extension of Meyers-Serrin's result - Theorem 3, p. 252, in \cite{evans}), the value function $v$ can be
approximated by a sequence of functions $\varphi^n \in C^{\infty}([0,T] \times {\mathcal O}, \mathbb{R}^+)$ such that $\varphi^n \rightarrow v$  in $W^{2,1,2}(\mathcal{Q}) \cap C([0,T], L^2({\mathcal O}))$ as $n \rightarrow  \infty$ and $\partial_t \varphi^n+\mathcal L\varphi^n$ is bounded.  Since $(\varphi^n)_{n \geq 1}$ are admissible test functions, they verify the constraint \eqref{constraint}. Therefore, using the assumptions on $m$ and passing to the limit, we derive that the value
function $v$ satisfies
\begin{align}\label{eqq1}
\int_{\mathcal O} v(0,x) m^*_0(dx) + \int_0^T \int_{\mathcal O}
\left\{\frac{\partial v}{\partial t} + \mathcal L
  v\right\}m_t(dx)dt \geq 0.
\end{align}
From the above inequality, we derive that
\begin{align}\label{eqqq}
\int_{\mathcal O} v(0,x) m^*_0(dx) \geq - \int_0^T \int_{\mathcal O}
\left\{\frac{\partial v}{\partial t} + \mathcal L
  v\right\}m_t(dx)dt.
\end{align}
Since $v$ satisfies the variational inequality $\eqref{ineq}$ and due to the positivity of $m$ and Assumption ($f$-Reg), we get
\begin{align*}
 - \int_0^T \int_{\mathcal O}
\left\{\frac{\partial v}{\partial t} + \mathcal L
  v\right\}m_t(dx)dt \geq \int_0^T \int_{\mathcal O} f(t,x)m_t(dx)dt.
\end{align*}
Combining the two above relations and by arbitrariness of $m \in \mathcal{A}(m_0^*)$, we get
\begin{align*}
\int_{\mathcal O} v(0,x) m^*_0(dx) \geq \sup_{m \in \mathcal{A}(m^*_0)} \int_0^T \int_{\mathcal O} f(t,x)m_t(dx)dt.
\end{align*}

Part ii. Since the stopping time $\bar{\tau}^{X_0}$ given by $\eqref{opt}$ is optimal for the stopping problem $\eqref{valuefct11}$, we derive that
\begin{align}\label{eqq}
\int_{\mathcal O} v(0,x) m^*_0(dx) = \int_0^T \int_{\mathcal O} f(t,x)m_t^*(dx)dt,
\end{align}
with $m^*$ defined by $m_t^*(A):=\int_{\mathcal O} m^*_0(dx) \mathbb{P}[X_t^{x} \in A,t<{\bar \tau}^{x}]$ for all $A \in \mathcal{B}({\mathcal O})$.\\
Using part i. and the fact that $m^* \in \mathcal{A}(m_0^*)$, the result follows.\\

Part iii. Let $m^*$ be defined in part ii. Since by the results above it is a maximizer,  we have \textcolor{black}{$\int_0^T\int_{\mathcal O} f(t,x)\bar{m}_t(dx)dt=\int_0^T\int_{\mathcal O} f(t,x)m^*_t(dx)dt$}. Therefore
\begin{align*}
0&=\int_0^T\int_{\mathcal O} f(t,x)\bigg(\bar{m}_t(dx)-m^*_t(dx)\bigg)dt \nonumber \\
& =\int_{v>0} f(t,x)\bar{m}_t(dx)dt+\int_{v=0} f(t,x)\bar{m}_t(dx)dt-\int_0^T\int_{\mathcal O} f(t,x)m^*_t(dx)dt \nonumber \\
& =\int_{v>0} \bigg( -\frac{\partial v}{\partial t} -\mathcal{L}v\bigg)\bar{m}_t(dx)dt-\int_0^T\int_{\mathcal O} f(t,x)m^*_t(dx)dt+\int_{v=0} f(t,x)\bar{m}_t(dx)dt, \nonumber \\
\end{align*}
where the last relation follows since $v$ satisfies the variational inequality $\eqref{ineq}$. Now, since $-\frac{\partial v}{\partial t} -\mathcal{L}v=0$ a.e. on $\{v=0\}$ and $m^*$ satisfies $\eqref{eqq}$, we get
\begin{align*}
0&=\int_0^T\int_{\mathcal O}\bigg( -\frac{\partial v}{\partial t} -\mathcal{L}v\bigg)\bar{m}_t(dx)dt-\int_0^T\int_{\mathcal O} f(t,x)m^*_t(dx)dt+\int_{v=0} f(t,x)\bar{m}_t(dx)dt \nonumber \\
&=\int_0^T\int_{\mathcal O}\bigg( -\frac{\partial v}{\partial t} -\mathcal{L}v\bigg)\bar{m}_t(dx)dt-\int_{\mathcal O} v(0,x) m^*_0(dx)+\int_{v=0} f(t,x)\bar{m}_t(dx)dt.
\end{align*}
Using the above relation, the inequality $\eqref{eqqq}$ and the fact that $f \leq 0$ a.e. on $\{v=0\}$, we finally obtain that
\begin{align*}
\int_{v=0} f(t,x)\bar{m}_t(dx)dt=0,
\end{align*}
and
\begin{align}\label{equation}
\int_0^T\int_{\mathcal O}\bigg( \frac{\partial v}{\partial t} +\mathcal{L}v\bigg)\bar{m}_t(dx)dt+\int_{\mathcal O} v(0,x) m^*_0(dx)=0.
\end{align}
Let us now show that \eqref{equation} implies that
\begin{align*}
\int_0^T\int_{{\mathcal O}}
\left\{\frac{\partial \phi}{\partial t} + \mathcal L
  \phi\right\}\bar{m}_t(dx)dt+\int_{\mathcal O} \phi(0,x) m^*_0(dx)=0,
\end{align*}
for all $C_c^\infty$ functions $\phi$ such that $\supp \phi \textcolor{black}{\subseteq} \{(t,x) \in [0,T] \times {\mathcal O}:\,\, v>0 \}$.

\textcolor{black}{First note that, by the same approximation procedure as the one used for the value function $v$  in Part i. (using an extension of Meyers-Serrin's result)}, any non-negative function $u$ in $W^{2,1,2}(\mathcal{Q})$ satisfies the constraint \eqref{eqq1}.

Let $\phi$ be a $C^\infty_c$ non-negative function such that $\supp \phi \subseteq \{(t,x) \in [0,T] \times {\mathcal O}:\,\, v>0 \}$. Up to an appropriate scale factor, one can assume that $\phi \leq v$. Suppose that
\begin{align}\label{equation1}
\int_0^T\int_{{\mathcal O}}
\left\{\frac{\partial \phi}{\partial t} + \mathcal L
  \phi\right\}\bar{m}_t(dx)dt+\int_{\mathcal O} \phi(0,x) m^*_0(dx)>0.
\end{align}
Subtracting $\eqref{equation1}$ from \eqref{equation}, we obtain that
\begin{align*}
\int_0^T\int_{{\mathcal O}}
\left\{\frac{\partial (v-\phi)}{\partial t} + \mathcal L
  (v-\phi)\right\}\bar{m}_t(dx)dt+\int_{\mathcal O} (v-\phi)(0,x) m^*_0(dx)<0,
\end{align*}
Since $v-\phi$ is a non-negative function belonging to $W^{2,1,2}(\mathcal{Q})$, we get a contradiction. This implies that for all non-negative $C^\infty_c$ functions $\phi$ such that $\supp \phi \subseteq \{(t,x) \in [0,T] \times {\mathcal O}:\,\, v>0 \}$ we have
\begin{align*}
\int_0^T\int_{{\mathcal O}}
\left\{\frac{\partial \phi}{\partial t} + \mathcal L
  \phi \right\}\bar{m}_t(dx)dt+\int_{\mathcal O} \phi(0,x) m^*_0(dx)=0.
\end{align*}
The result can be extended to an arbitrary $C^\infty_c$ function $\phi$ (which also takes negative values) such that $\supp \phi \subseteq \{(t,x) \in [0,T] \times {\mathcal O}:\,\, v>0 \}$. Using appropriate scaling factors and similar arguments as above,  one can show that $\int_0^T\int_{{\mathcal O}}
\left\{\frac{\partial \phi}{\partial t} + \mathcal L
  \phi \right\}\bar{m}_t(dx)dt+\int_{\mathcal O} \phi(0,x) m^*_0(dx) < 0$ and $\int_0^T\int_{{\mathcal O}}
\left\{\frac{\partial \phi}{\partial t} + \mathcal L
  \phi \right\}\bar{m}_t(dx)dt+\int_{\mathcal O} \phi(0,x) m^*_0(dx) > 0$ cannot be satisfied.
Hence, for all $C_c^\infty$ functions $\phi$ such that $\supp \phi \subseteq \{(t,x) \in [0,T] \times {\mathcal O}:\,\, v>0 \}$, we have
\begin{align}
\int_0^T\int_{{\mathcal O}}
\left\{\frac{\partial \phi}{\partial t} + \mathcal L
  \phi\right\}\bar{m}_t(dx)dt+\int_{\mathcal O} \phi(0,x) m^*_0(dx)=0.
\end{align}

\end{proof}

\textcolor{black}{We now illustrate the relation between the relaxed and strong formulation of  the optimization problem in the MFG context, as well as the relation with the mixed solutions introduced in \cite{bertucci2017optimal}}. 

\begin{theorem}\textbf{[MFG optimal stopping problem]}
\textcolor{black}{Let Assumptions (X-SDE), (X-PDE), (X-Reg) and
  ($f$-Exist-MFG) be satisfied.} Let $m^*$ be a Nash equilibrium of the relaxed MFG problem and let $v$ be the value function of the optimal stopping problem
\begin{align*}
v(t,x)=\textcolor{black}{\sup_{\textcolor{black}{\tau \in {\mathcal{T}^{t}_W([t,T])}}}} \mathbb{E} \left[\int_t^{\tau \wedge \tau_{\mathcal O}^{(t,x)}} f(s,X_s^{(t,x)}, m_s^*)ds\right],
\end{align*}
with $(t,x) \in [0,T] \times \mathbb{R}$ and $\tau_{{\mathcal O}}^{(t,x)}:=\inf \{ s \geq t: \,\,\, X_s^{(t,x)} \notin {\mathcal O} \}$.\\
We have
\begin{enumerate}[i.]
\item   \textcolor{black}{\textit{Relation with the strong formulation}}

$\int_{\mathcal O} v(0,x)m^*_0(dx)=\int_0^T \int_{\mathcal O} f(s,x,m_s^*)m^*_s(dx)ds.$

\item \textcolor{black}{\textit{Relation with mixed solutions}}

$m^*$ satisfies
\begin{enumerate}[a.]
\item  $\int_{\mathcal{S}}f(t,x, m_t^*)m^*_t(dx)dt=0$, with $\mathcal{S}:=\{(t,x) \in [0,T] \times {\mathcal O}: \,\, v(t,x)=0\}.$
\item For all $C^\infty_c$ functions $\phi$ such that $\supp \phi \subseteq \{(t,x) \in [0,T] \times {\mathcal O}:\,\, v>0 \}$, the following holds
\begin{align*}
\int_0^T\int_{{\mathcal O}}
\left\{\frac{\partial \phi}{\partial t} + \mathcal L
  \phi\right\}m_t^*(dx)dt+\int_{\mathcal O} \phi(0,x) m^*_0(dx)=0.
\end{align*}
\end{enumerate}
\end{enumerate}
\end{theorem}
\begin{proof}
The proof follows by using the results obtained in Theorem \ref{TH}
applied to the instantaneous reward function $f(\cdot,m^*)$ (which
satisfies Assumption \textcolor{black}{($f$-Exist)(ii)}, so that Theorem \ref{TH} can be applied), together with the Nash equilibrium property of $m^*$.
\end{proof}

\begin{remark}
It follows from the variational inequality $\eqref{ineq}$ that $f\leq 0$ on $\{v=0\}$. Therefore,
if $f\neq 0$ on $\{v=0\}$, $\int_{\{v=0\}}m^*_t(dx)\, dt = 0$. Such a solution is
called a \textit{pure solution} in \cite{bertucci2017optimal}, meaning that the
agent will exit the game immediately upon entering the exercise
region.
\end{remark}

\section{\textcolor{black}{Fixed-point algorithm and convergence in the case of potential games}}

 \textcolor{black}{We first show that, in the case of  \textit{potential games}, the search
for MFG equilibrium reduces to the maximization of a functional. The
reward function of a potential game satisfied the following
assumption.
\begin{assumption}[\textbf{$f$-Pot}]
The reward function is of the form
$$
f(t,x,m) = \sum_{i=1}^K \bar f_i \left(t,g_i *m_t\right) g_i(x),
$$
where for each $i$, $\bar f_i$ is bounded, measurable in $t$, and continuous
and decreasing in the second argument, and $g_i\in C^2(\mathcal O; \mathbb{R})$ such that $g_i$, $|\!|\nabla_{_X} g_i|\!|$ and
  $|\!|H_{_X} g_i|\!|$ are bounded. Moreover,  for each $i$, there exists $\overline F_i:[0,T] \times \mathbb{R} \mapsto \mathbb{R}$ such that $\partial_x \overline F_i(t,x) = \bar
f_i(t,x)$ and $\overline F_i(\cdot,0) \in L^1([0,T])$.
\end{assumption}} 

\begin{proposition}\label{Potential}
Let Assumption ($f$-Pot) be satisfied.
Then $m^*\in \mathcal
A(m^*_0)$ is a Nash
equilibrium of the relaxed optimal stopping problem if and only if
$$
F(m^*) = \sup_{m\in \mathcal A(m_0^*)} F(m),
$$
where
$$
F(m) = \sum_{i=1}^K \int_0^T \overline F_i(t,g* m_t) dt,
$$
\end{proposition}
\begin{proof}
Assume that $m^*$ is a Nash equilibrium. By definition we then have
$$
\int_0^T \sum_i \partial_x \overline F_i(t,g*m^*_t) (g*m_t - g*m^*_t)
dt\leq 0.
$$
Since $\bar f_i$ is decreasing in the second argument, $\bar F_i$ is
concave in the second argument, and by concavity this implies that $F(m^*)\geq F(m)$. Conversely,
assume that $m^*$ is a maximizer of $F$. For every $\alpha\in[0,1]$
and every $m\in \mathcal A(m^*_0)$,
then,
$$
\sum_{i=1}^K \int_0^T \left\{\overline F_i(t,g*m^*_t) - \overline F_i
  (t,\alpha g* m_t + (1-\alpha) g* m^*_t )\right\}dt \geq 0,
$$
which implies that
$$
\sum_{i=1}^K \int_0^T \partial_x \overline F_i(t,\xi_t) g*(m^*_t-m_t)dt \geq 0,
$$
where $\xi_t \in [g*m^*_t, \alpha g* m_t + (1-\alpha) g*
m^*_t]$. Making $\alpha$ tend to $0$ and using the dominated
convergence theorem, we conclude that
$$
\sum_{i=1}^K \int_0^T \bar f_i(t,g*m^*_t) g*(m^*_t-m_t)dt \geq 0.
$$
\end{proof}

\textcolor{black}{We propose now a fixed-point algorithm for \textit{potential
  games}. We  use
the notations of Proposition \ref{Potential}.}\\

\textbf{Algorithm}
\begin{itemize}
\item Fix $m^0 \in \mathcal{A}(m_0^*);$
\item For $k=0$ to $N$
\begin{itemize}
\item[$\bullet$] Compute $u^{k}$ the solution of the obstacle problem \eqref{ineq} associated with $f(\cdot, m^{k})$;
\item[$\bullet$] Let $\Tilde{m}^{k} \in \mathcal{A}(m_0^*)$ be such that
$\Tilde{m}_t^{k}(A)=\int_{\mathcal O} m^*_0(dx)\mathbb P[X_t^{x} \in A; \,\, t<\tau_k^{x}]$, for all $A \in \mathcal{B}({\mathcal O})$, where $\tau_k^{x}:=\inf \{0 \leq t \leq T:\,\, u^k(t,X^{x}_t)=0\}$\footnote{Note that, for each $k$ from $0$ to $N$, we extend $u^k$ such that $u^k(t,x)=0$ for all $t \in [0,T]$ and $x \notin \mathcal{O}$. Therefore, we have $\tau_k^{x} \leq \tau_\mathcal{O}^{x} \wedge T$ a.s.};
\item[$\bullet$] Let $\rho^{k}$ be a maximizer of  $\rho \mapsto F(m^k+\rho(\Tilde{m}^k-m^k));$
\item[$\bullet$] Set $m^{k+1}:=m^{k}+\rho^k(\Tilde{m}^k-m^{k})$;
\item[$\bullet$] Set $k \leftarrow k+1$.
\end{itemize}
\end{itemize}
In the above algorithm, $N$ represents the number of iterations.\\

For each $m \in \mathcal{A}(m^*_0)$, define
\begin{align*}
\mathcal{C}(m):=& \bigg\{ m+\hat{\rho}(m^*-m), \,\,\,  m^* \in \underset{m' \in \mathcal{A}(m^*_0)}{\arg \max}  \int_0^T \int_{ {\mathcal O}} f(m_{\textcolor{black}{t}})m_{\textcolor{black}{t}}'\textcolor{black}{(}dx\textcolor{black}{)}dt,\,\,\bigg. \nonumber \\ & \bigg. \hat{\rho}\in \underset{\rho \in [0,1]}{\argmax}\,\, F(m+\rho(m^*-m)) \bigg\}.
\end{align*}

\begin{lemma}\label{closed}
\textcolor{black}{Let Assumptions (X-SDE), (X-PDE), (X-Reg) and ($f$-Pot) be satisfied.} The set{-}valued map $m \in \mathcal{A}(m_0^*) \mapsto \mathcal{C}(m)$ has a closed graph and $\bar{m} \in \mathcal{A}(m^*_0)$ is a relaxed Nash equilibrium if and only if it satisfies $\bar{m} \in \mathcal{C}(\bar{m})$.
\end{lemma}
\begin{proof}

Let $(m^{n})_{n\ge 1} \in \mathcal{A}(m_0^*)$ be a sequence converging weakly to some $\hat{m} \in \mathcal{A}(m_0^*)$ and $\bar{m}^{n}=m^{n}+\hat{\rho}^{n}(m^{n^*}-m^{n})$ such that $\bar{m}^{n} \in \mathcal{C}(m^{n})$ weakly converging to some $\bar{m}$. Let us prove that $\bar{m} \in \mathcal{C}(\hat{m})$. Taking subsequences if necessary, we can assume that $\hat{\rho}^{n}$ converges to some $\hat{\rho} \in [0,1]$ and $m^{n^*}$ weakly converges to some $m^*$.

Since $m^{n^*}$ maximizes the map $m \mapsto \int_0^T \int_{\mathcal O} f(m^{n}_{{t}})m_{t}(dx)dt$, we get that
\begin{align*}
\int_0^T \int_{\mathcal O} f(m^{n}_{{t}})m^{n^*}_{{t}}(dx)dt \geq \int_0^T \int_{\mathcal O} f(m^{n}_{{\textcolor{black}{t}}})m_{{t}}(dx)dt,
\end{align*}
for all $m \in \mathcal{A}(m_0^*)$. For simplicity, we consider here the case $K=1$ and drop the index $i$. Using the same arguments as \textcolor{black}{those} in the proof of Theorem \ref{exist}, we may say that, up to taking subsequences, the sequence $(g*m^{\textcolor{black}{n}})_{n\geq 1}$ (resp. $(g*m^{\textcolor{black}{n}^*})_{n\geq 1}$) converges
in $L^1([0,T])$ to $g*\hat{m}$ (resp. $g*m^*$).
Due to the continuity of $\bar{f}$, we derive that $\bar{f}(t,g*m^{\textcolor{black}{n}}_{{\textcolor{black}{t}}})g*m^{\textcolor{black}{n}^*}_{{\textcolor{black}{t}}}$ (resp.  $\bar{f}(t, g*m^{\textcolor{black}{n}}_{{\textcolor{black}{t}}})g*m_{\textcolor{black}{t}}$) converges for a.e. $t$ to $\bar{f}(t,g*\hat{m}_{\textcolor{black}{t}})g*m^*_{\textcolor{black}{t}}$ (resp.  $\bar{f}(t,g*\hat{m}_{\textcolor{black}{t}})g*m_{\textcolor{black}{t}}$). By Corollary \ref{bounded.cor}, $g*
m^{\textcolor{black}{n}^*}$ is uniformly bounded, therefore, by appealing to the dominated \textcolor{black}{convergence} theorem, we derive
\begin{align*}
\int_0^T \int_{\mathcal O} f(\hat{m}_{\textcolor{black}{t}})m_{\textcolor{black}{t}}^*\textcolor{black}{(}dx\textcolor{black}{)}dt \geq \int_0^T \int_{\mathcal O} f(\hat{m}_{\textcolor{black}{t}})m_{\textcolor{black}{t}}\textcolor{black}{(}dx\textcolor{black}{)}dt,
\end{align*}
for all $m \in \mathcal{A}(m_0^*)$, that is
\begin{align*}
m^* \in \underset{m \in \mathcal{A}(m^*_0)}{\arg \max}  \int_0^T \int_{ {\mathcal O}} f(\hat{m}_{\textcolor{black}{t}})m_{\textcolor{black}{t}}\textcolor{black}{(}dx\textcolor{black}{)}dt.
\end{align*}
Now it remains to show that $\hat{\rho}$ is a maximizer of $\rho \mapsto F(\hat{m}+\rho(m^*-\hat{m}))$. For each $n$, we have
$F(m^{\textcolor{black}{n}}+\hat{\rho}^{\textcolor{black}{n}}(m^{\textcolor{black}{n}^*}-m^{\textcolor{black}{n}})) \geq F(m^{\textcolor{black}{n}}+\rho(m^{\textcolor{black}{n}^*}-m^{\textcolor{black}{n}}))$, for all $\rho \in [0,1]$, for all $n$. Taking the limit $n \rightarrow \infty$ and using similar arguments as above, as well as the assumptions on $F$, we get
$$F(\hat{m}+\hat{\rho}(m^*-\hat{m})) \geq F(\hat{m}+\rho(m^*-\hat{m})), \,\, {\rm for \,\, all\,\,} \rho \in [0,1].$$ To conclude, we have $\bar{m} \in \mathcal{C}(\hat{m})$.

It is clear that, if  $m \in \mathcal{A}(m_0^*)$ is a relaxed Nash equilibrium, then it satisfies $m \in \mathcal{C}(m)$. Conversely, one can show that if $m \in \mathcal{C}(m)$, then $m$ corresponds to a relaxed Nash equilibrium. Indeed, if $m \in \mathcal{C}(m)$, then we have $\hat{\rho}=0$ or $m^*=m$. If $\hat{\rho} = 0$, then $\int_0^T \int_{\mathcal O}  f(m_{\textcolor{black}{t}})(m_{\textcolor{black}{t}}^*\textcolor{black}{(}dx\textcolor{black}{)}-m_{\textcolor{black}{t}}\textcolor{black}{(}dx\textcolor{black}{)})dt\leq 0.$ Since $m^*$ is a maximizer of the map $m' \mapsto \int_0^T \int_{ {\mathcal O}} f(m_{\textcolor{black}{t}})m_{\textcolor{black}{t}}'\textcolor{black}{(}dx\textcolor{black}{)}dt$, we derive that
$\int_0^T \int_{\mathcal O} f(m_{\textcolor{black}{t}})(m_{\textcolor{black}{t}}^*\textcolor{black}{(}dx\textcolor{black}{)}-m_{\textcolor{black}{t}}\textcolor{black}{(}dx\textcolor{black}{)})dt=0,$ which implies that $m$ corresponds to a relaxed Nash  equilibrium. If $m^*=m$, the conclusion is clear.
\end{proof}

We now give the following convergence result.

\begin{theorem}
\textcolor{black}{Let Assumptions (X-SDE), (X-PDE) , (X-Reg) and ($f$-Pot) be satisfied.} Then the cluster points of the sequence $(m^n)_{\textcolor{black}{n\ge 1}}$ generated by the previous
algorithm belong to the set of relaxed Nash equilibria and the
sequence  $(u^n(0,x))_{\textcolor{black}{n\ge 1}}$ converges for all $x \in {\mathcal O}$ to $\bar
u(0,x)$, the value function of the obstacle problem associated with
cost functional $f(\cdot,\bar m)$, where $\bar m$ is a relaxed Nash
equilibrium.
\end{theorem}
\begin{proof}
First note that, by using the definition of $\Tilde{m}^n$ and Theorem $\ref{TH}$ \textcolor{black}{part ii.}, we get that $\Tilde{m}^n \in \underset{m' \in \mathcal{A}(m^*_0)}{\arg \max}  \int_0^T \int_{ {\mathcal O}} f(m^n_{\textcolor{black}{t}})m_{\textcolor{black}{t}}'\textcolor{black}{(}dx\textcolor{black}{)}dt$. We thus have $m^{\textcolor{black}{{n+1}}} \in \mathcal{C}(m^{\textcolor{black}{{n}}})$, for all $n$.

 Let $(m^{k_n})_{\textcolor{black}{n\ge 1}}$ be a sequence converging weakly to some $m$, and taking a subsequence again if necessary, we may also assume that $m^{k_n+1}$ converges to some $m_1$. As by the previous theorem the set\textcolor{black}{-}valued map $m \in \mathcal{A}(m_0^*) \mapsto \mathcal{C}(m)$ has a closed graph, we have $m_1 \in \mathcal{C}(m)$, that is $m_1=m+\hat{\rho}(m^*-m)$, with $m^* \in \underset{m' \in \mathcal{A}(m^*_0)}{\arg \max} \int_0^T \int_{\mathcal O} f(m_{\textcolor{black}{t}})m_{\textcolor{black}{t}}'\textcolor{black}{(}dx\textcolor{black}{)}dt$  and $\hat{\rho} \in \underset{\rho \in [0,1]}{\arg \max}\,\,F(m+\rho(m^*-m)).$

Now, since the sequence $(F(m^n))_{\textcolor{black}{n\ge 1}}$ is increasing, one has $F(m)=F(m_1).$ Assume now that $m$ is not a Nash equilibrium, that is $m \notin \underset{m' \in \mathcal{A}(m^*_0)}{\arg \max} \int_0^T \int_{\mathcal O} f(m_{\textcolor{black}{t}})m_{\textcolor{black}{t}}'\textcolor{black}{(}dx\textcolor{black}{)}dt$. Therefore, $\int_0^T \int_{\mathcal O} f(m_{\textcolor{black}{t}})(m_{\textcolor{black}{t}}^*\textcolor{black}{(}dx\textcolor{black}{)}-m_{\textcolor{black}{t}}\textcolor{black}{(}dx\textcolor{black}{)})dt>0$. Moreover, using Lemma $\ref{closed}$, we have $m \notin \mathcal{C}(m)$ which implies that $\hat{\rho}>0$. Hence,  we conclude that $F(m_1)=F(m+\hat{\rho}(m^*-m))>F(m)$, which represents a contradiction.

Let us now prove the convergence of the sequence $(u^n(0,x))_{\textcolor{black}{n\ge 1}}$ for all $x \in {\mathcal O}$.

Since all Nash equilibria $m$ lead to the same value (see Theorem \ref{Uniqueness}), we can define $\bar{u}$ as being the solution of the obstacle problem associated with $f(\cdot,\bar{m})$, with $\bar{m}$ a Nash equilibrium.

Let $u^{k_n}$ be a given subsequence. Up to \textcolor{black}{subtracting} a subsequence again, one can assume that $m^{k_n}$ converges weakly to some $m^* \in \mathcal{A}(m_0^*)$, which, by the results above, is a relaxed Nash equilibrium.

Fix $x \in \mathcal{O}$. We have:
\begin{align*}
|u^{k_n}(0,x)-\bar{u}(0,x)| &\leq \sup_{\tau \in \textcolor{black}{\mathcal{T}^W([0,T])}} \mathbb{E}\left[\int_0^{\tau \wedge \tau_{\mathcal O}^{\textcolor{black}{x}}} \left|f(s,X_s^{\textcolor{black}{x}},m_s^{k_n})-f(s,X_s^{\textcolor{black}{x}},m_s^*)\right|ds\right] \nonumber \\
& \leq \mathbb{E}\left[\int_0^T \left|f(s,X_s^{\textcolor{black}{x}},m_s^{k_n})-f(s,X_s^{\textcolor{black}{x}},m_s^*)\right|ds\right].
\end{align*}
Using again the convergence in $L^1([0,T])$ of $g*m^{k_n}$ to $g*m^*$, the assumptions on $f$ together with Lebesgue Theorem, we get that the last term of the above inequality converges to 0. We can conclude that from every subsequence of $u^{k_n}(0,x)$, we can extract a further subsequence which converges to $\bar{u}(0,x)$. The result follows.
\end{proof}

\section{Acknowledgement}
Peter Tankov gratefully acknowledges financial support from the LABEX
ECODEC (ANR-11-IDEX-0003/LabexEcodec/ANR-11-LABX-0047) and from the
FIME Research Initiative.
\appendix
\section{Appendix}

We show here that the representation \eqref{valuefct} remains true when the initial condition $\xi$ is random. More precisely, we have the following result.
\begin{lemma}\label{repres}
Let $\xi \in \textcolor{black}{L}^2({\mathcal O}, \mathcal{F}_0)$. Then we have
\begin{align}\label{res}
v(0,\xi)=\underset{\tau \in \mathcal{T}([0,T])}\esssup\,\, \mathbb{E}\left[\int_0^{\tau \wedge \tau_{\mathcal O}^{\textcolor{black}{\xi}}}f(s,X_s^{\textcolor{black}{\xi}})ds | \mathcal{F}_0 \right] {\rm \,\, a.s.}
\end{align}
\end{lemma}
\begin{proof} The proof is based \textcolor{black}{on quite classical} arguments \textcolor{black}{and} we give it here for the reader's \textcolor{black}{convenience}. Let us  first consider a simple random variable $\xi^n \in \textcolor{black}{L}^2({\mathcal O}, \mathcal{F}_0)$, \textcolor{black}{being} such that there exists $n \in \mathbb{N}$, $A_1, A_2,...,A_n \in \mathcal{F}_0$ \textcolor{black}{and} $x_1,x_2,...,x_n \in {\mathcal O}$ such that
\begin{align}\label{simple}
\xi^n=\sum_{i=1}^n x_i \textbf{1}_{A_i} {\rm \,\, a.s.}
\end{align}
By using the definitions of $\xi^n$ and $v(t,x)$, we obtain
\begin{align*}
\underset{\tau \in \mathcal{T}([0,T])}\esssup\,\, \mathbb{E}\left[\int_0^{\tau \wedge \tau_{\mathcal O}^{\textcolor{black}{\xi^n}}}f(s,X_s^{\textcolor{black}{\xi^n}})ds | \mathcal{F}_0\right]&=\sum_{i=1}^n \textbf{1}_{A_i} \underset{\tau \in \textcolor{black}{\mathcal{T}^W([0,T])}}\sup\,\, \mathbb{E}\left[\int_0^{\tau \wedge \tau_{\mathcal O}^{\textcolor{black}{x_i}}}f(s,X_s^{\textcolor{black}{x_i}})ds \right] \nonumber \\
&=\sum_{i=1}^n \textbf{1}_{A_i} v(0,x_i)=v(0,\xi^n) {\rm \,\,a.s.}
\end{align*}
Now, in the general case, we approximate $\xi$ \textcolor{black}{by} a sequence of simple random variables $\xi^n$ of the form given by \eqref{simple}. The continuity of $v$ with respect to $x$ implies that
\begin{align}\label{first}
v(0,\xi^n) \textcolor{black}{\rightarrow} v(0, \xi) {\rm \,\, a.s.\,\, as\,\,} n\textcolor{black}{\rightarrow} \infty.
\end{align}

We have
\begin{align*}
&\left| \underset{\tau \in \mathcal{T}([0,T])}\esssup\,\, \mathbb{E}\left[\int_0^{\tau \wedge \tau_{\mathcal O}^{\textcolor{black}{\xi^n}}}f(s,X_s^{\textcolor{black}{\xi^n}})ds | \mathcal{F}_0\right]-\underset{\tau \in \mathcal{T}([0,T])}\esssup\,\, \mathbb{E}\left[\int_0^{\tau \wedge \tau_{\mathcal O}^{\textcolor{black}{\xi}}}f(s,X_s^{\textcolor{black}{\xi}})ds | \mathcal{F}_0\right] \right|  \nonumber\\
& \leq  \mathbb{E}\left[\int_0^{  \tau_{\mathcal O}^{\textcolor{black}{\xi}} \wedge \tau_{\mathcal O}^{\textcolor{black}{\xi^n}}}\left| f(s,X_s^{\textcolor{black}{\xi^n}})-f(s,X_s^{\textcolor{black}{\xi}})\right|ds | \mathcal{F}_0\right]+ \mathbb{E}\left[\int_{ \tau_{\mathcal O}^{\textcolor{black}{\xi^n}} \wedge\tau_{\mathcal O}^{\textcolor{black}{\xi}} }^{\tau_{\mathcal O}^{\textcolor{black}{\xi^n}}}|f(s,X_s^{\textcolor{black}{\xi^n}})|ds | \mathcal{F}_0\right] \nonumber \\
& +\mathbb{E}\left[\int_{ \tau_{\mathcal O}^{\textcolor{black}{\xi^n}} \wedge\tau_{\mathcal O}^{\textcolor{black}{\xi}} }^{\tau_{\mathcal O}^{\textcolor{black}{\xi}}}|f(s,X_s^{\textcolor{black}{\xi}})|ds | \mathcal{F}_0\right] \,\,\, {\rm a.s.}
\end{align*}
Since $\xi^n \textcolor{black}{\rightarrow} \xi$ a.s. as $n \textcolor{black}{\rightarrow} \infty$, we get that $\tau_{\mathcal O}^{\textcolor{black}{\xi^n}} \textcolor{black}{\rightarrow} \tau_{\mathcal O}^{\textcolor{black}{\xi}}$ a.s. as $n \textcolor{black}{\rightarrow} \infty$ due to the continuity property of the first passage time for elliptic diffusions (see \textcolor{black}{Proposition 4.4. in \cite{pardoux1998backward}}). Using the continuity property of the solution of the SDE with respect to the initial condition, together with \textcolor{black}{the assumptions on $f$ and Lebesgue Theorem}, it follows that
\begin{align}\label{second}
v(0,\xi^n) \textcolor{black}{\rightarrow} \underset{\tau \in \mathcal{T}([0,T])}\esssup\,\, \mathbb{E}\left[\int_0^{\tau \wedge \tau_{\mathcal O}^{\textcolor{black}{\xi}}}f(s,X_s^{\textcolor{black}{\xi}})ds | \mathcal{F}_0\right] {\rm \,\, a.s\,\, as\,\, } n \textcolor{black}{\rightarrow} \infty.
\end{align}
By \eqref{first} and \eqref{second} and the uniqueness of the limit, we get \eqref{res}.
\end{proof}

\bibliography{BibGB}
\bibliographystyle{agsm}
\end{document}